\documentclass[a4paper,11pt]{amsart}

\usepackage{amssymb}
\usepackage[latin1]{inputenc}

\usepackage[all]{xy}
 \usepackage{hyperref} 

 \usepackage[left=2cm,right=2cm, top=3cm, bottom=3cm]{geometry}

\newtheorem{theorem}{Theorem}[section]
\newtheorem{lemma}[theorem]{Lemma}
\newtheorem{corollary}[theorem]{Corollary}
\newtheorem{proposition}[theorem]{Proposition}

\newtheorem{thmx}{Theorem}

\theoremstyle{remark}
\newtheorem{remark}[theorem]{Remark}

\theoremstyle{definition}
\newtheorem{definition}[theorem]{Definition}

\newcommand{\Aut}{\mathrm{Aut}}
\newcommand{\Stab}{\mathrm{Stab}}
\newcommand{\ord}{\mathrm{ord}}
\newcommand{\Fix}{\mathrm{Fix}}
\newcommand{\Cl}{\mathrm{Cl}}
\newcommand{\Alb}{\mathrm{Alb}}
\newcommand{\im}{\mathrm{Im}}
\newcommand{\Sing}{\mathrm{Sing}}
\newcommand{\Inn}{\mathrm{Inn}}      
\newcommand{\Orb}{\mathrm{Orb}}

\title{On Semi-isogenous mixed surfaces} 																											
\makeatletter

\@addtoreset{equation}{section}
\makeatother
\author[N. Cancian]{Nicola Cancian}

\author[D. Frapporti]{Davide Frapporti}
\thanks{The authors thank F. Catanese, C. Glei\ss ner and R. Pignatelli for inspiring 
conversations and discussions.\\
Some of the results contained in this paper were developed in Spring 2015 when the 
first author was visiting the University of Bayreuth, he is deeply grateful for the friendly and stimulating atmosphere he found.
The present work took mainly place in the realm of the DFG
Forschergruppe 790 ``Classification of algebraic
surfaces and compact complex manifolds''.
The second author is member of  G.N.S.A.G.A. of I.N.d.A.M.
and acknowledges support of the ERC-advanced Grant 340258-TADMICAMT}

\keywords{Surfaces of general type, finite group actions} 
\subjclass[2000]{14J29, 14L30, 14Q10} 

\date{\today}

\begin{document}
\begin{abstract}
Let $C$ be a smooth projective curve and $G$ a finite subgroup
of $\Aut(C)^2\rtimes \mathbb Z_2$  whose action is \textit{mixed}, i.e.~there are elements in $G$ exchanging the two
isotrivial fibrations of $C\times C$.  
Let  $G^0\triangleleft G$ be the index two subgroup $G\cap\Aut(C)^2$.
If $G^0$ acts freely, then $X:=(C\times C)/G$ is smooth and we call it \textit{semi-isogenous mixed surface}.
In this paper we give an algorithm to determine semi-isogenous mixed surfaces
with given geometric genus, irregularity and self-intersection of the canonical class.
As an application we classify irregular semi-isogenous mixed surfaces with $K^2>0$ and  geometric genus equal to the irregularity; the regular case is subjected to some computational restrictions.

In this way we construct new examples of surfaces of general type with $\chi=1$.
We provide an  example of a minimal surface of general type with $K^2=7$ and $p_g=q=2$.
\end{abstract}
\maketitle

\section*{Introduction}

The study of compact complex projective surfaces of general type is a classical and long-standing research topic.
Despite the intensive effort made to improve our knowledge about surfaces of general type, their classification is still
an open problem. Nevertheless, there are some important inequalities, holding for minimal surfaces of general type $S$, 
 that allow us to deal with them more easily: 
$K_S^2\geq 1$, $\chi(\mathcal O_S)\geq 1$,  the Bogomolov-Miyaoka-Yau inequality $K_S^2\leq 9\chi(\mathcal O_S)$, 
and the Noether inequality $K_S^2\geq 2\chi(\mathcal O_S)-6$;
here and in the following we use the standard notation of the theory of the complex 
surfaces, as in \cite{Beau83, BHPV}.
 For motivation and for the state of the art (few years ago)  we suggest to the reader the survey \cite{CSGT}.

Since up to now a complete classification  seems out of reach, one 
 tries first to understand and classify the boundary cases, which usually are more interesting,
 for example $\chi(\mathcal O_S)=1$.
If this is the case, we have $1=\chi(\mathcal O_S)=1-q(S)+p_g(S)$ and so $p_g(S)=q(S)$.

By the main theorem of \cite{be82}, if $S$ is a minimal surface of general type
with $p_g(S)=q(S)\geq4$, then $S$ is the product of two genus 2 curves and $p_g(S)=q(S)=4$.
The case $p_g(S)=q(S)=3$ is classified too (see \cite{CCML98, Pir02, HP02}):
here either $S$ is  the symmetric square of a genus three curve or 
$S=(F_2 \times F_3)/\nu$ where $F_g$ is a curve of genus $g$ and $\nu$  is an involution  
acting on $F_2$ as an elliptic involution and freely on $F_3$.
It seems that the classification  becomes more complicated as the value of $p_g$ decreases, 
and the case $p_g = q \leq 2$ is still rather unknown.

 In the last years there has been growing interest in those surfaces
birational to the quotient of the product of two curves by the action of a  finite group
and  several new surfaces of general type with $p_g=q$ have been constructed in this way: see
\cite{BC04, BCG08, BCGP12, BP12, BP15,Frap13, FP15} for $p_g=0$, 
\cite{CP09, Pol07, Pol09, MP10,Frap13, FP15} for $p_g=1$,
\cite{pen11, zuc03} for $p_g=2$.  

In all of these articles the authors assume  that the group action  is free outside of a finite set of points.
We call this case {\it quasi-\'etale} since the induced map onto the quotient is quasi-\'etale in the sense of \cite{Cat07}.
 This includes the \textit{unmixed} case, which means that the group action on the product is diagonal,
induced by actions on the factors.

In the present paper we work in a different framework, without the quasi-\'etale assumption.
We consider the following situation.
Let $C$ be a smooth projective curve of genus $g(C)$ and $G$ a finite subgroup
of $\Aut(C)^2\rtimes \mathbb Z_2$  whose action is \textit{mixed}, i.e.~there are elements in $G$ exchanging the two
natural isotrivial fibrations of $C\times C$.  
The quotient surface $X:=(C\times C)/G$ is a \textit{mixed quotient}, and its 
minimal resolution of the singularities $S\to X$ is a \textit{mixed surface}.
We denote by  $G^0\triangleleft G$ the index two subgroup $G\cap\Aut(C)^2$, 
i.e.~the subgroup consisting of those elements that do not exchange the factors.

In general the singularities of $X$ are rather complicated. Assuming
that $G^0$ acts freely, i.e.~$(C\times C)/G^0$ is a surface 
isogenous to a product (cf.~\cite{Cat00}),  then $X$ is smooth  and we call it a \textit{semi-isogenous mixed surface}.
One aim of this paper is to give a systematic way  to classify these surfaces.

Following the strategies of the above mentioned papers, our classification method combines geometry and 
group theory. To each semi-isogenous mixed surface we can associate the 
group $G$ and a generating vector for  $G^0$ (see Definition \ref{gv}). The idea is that 
the geometry of $X$ is encoded in this pair of algebraic data, hence
the problem of constructing surfaces is translated into the problem of finding pairs (groups, generating vector)
subjected to conditions of combinatorial type.

The main result of this paper is an algorithm which, once the integers $p_g$, $q$ and $K^2$ are given, produces all 
semi-isogenous mixed surfaces with those invariants. 
We implemented the  algorithm using the computer algebra software
MAGMA \cite{MAGMA}; the script is available from
\[\mbox{\url{http://www.science.unitn.it/~cancian/}}\,\] 

As an application, running the program for all possible positive values of $K^2$ and $p_g=q$,  we obtained the following Theorems \ref{thmA}, \ref{thmB} and \ref{thmC}.
Note that the program  works for arbitrary values of $K^2$, $p_g$ and $q$, so more surfaces may be produced with it.

\begin{thmx}\label{thmA}
Let $X:=(C\times C)/G$ be a  semi-isogenous mixed surface  with  $p_g(X)=q(X)=0$ and $K^2_X>0$,
such that $|G^0|\leq 2000, \neq 1024 $.
Then either $X$ is $\mathbb P^2$, or $X$  belongs to one of the  $15$ families 
collected in Table $\ref{q0}$ and is of general type.
\end{thmx}

\begin{table}[!ht]
\caption{$p_g=q=0$, $K^2>0$}\label{q0}
{\small
\begin{tabular}{c|c|c|c|c|c|c|c|c|c}
$K_X^2$&$G$ & $Id(G)$ & $G^0$ & $Id(G^0)$ & $g(C)$ & Type & Branch Locus $\mathcal B$ &  $H_1(X,\mathbb Z)$&  min? \\
\hline
8 & $D_{2,8,5}\rtimes\mathbb{Z}_2^2$ & 64, 92& $\mathbb{Z}_2^2\times D_4$ & 32, 46&9&[0;$2^5$]& $\emptyset$ &$ \mathbb Z_2^3\times\mathbb Z_8$ &Yes\\
8 & & 256, 3679& $(\mathbb{Z}_2^3\rtimes \mathbb{Z}_4)\rtimes \mathbb{Z}_4$ & 128, 36&17&[0;$4^3$]& $\emptyset$ &$\mathbb Z_2^2\times\mathbb Z_4^2$ &Yes\\
8 & & 256, 3678& $(\mathbb{Z}_2^3\rtimes \mathbb{Z}_4)\rtimes \mathbb{Z}_4$ & 128, 36&17&[0;$4^3$]& $\emptyset$ &$\mathbb Z_2^2\times\mathbb Z_4^2$ &Yes\\
8 & & 256, 3678& $(\mathbb{Z}_2^3\rtimes \mathbb{Z}_4)\rtimes \mathbb{Z}_4$ & 128, 36&17&[0;$4^3$]& $\emptyset$ &$ \mathbb Z_2^4\times\mathbb Z_4$ &Yes\\
8 & & 256, 3678& $(\mathbb{Z}_2^3\rtimes \mathbb{Z}_4)\rtimes \mathbb{Z}_4$ & 128, 36&17&[0;$4^3$]& $\emptyset$ &$ \mathbb Z_4^3$ &Yes\\
\hline
6& $\mathbb{Z}_8\rtimes\mathbb{Z}_2^2$& 32,43& $\mathbb{Z}_2\times D_4$ & 16,11& 9 & [0; $2^6$]& $(3,-8)$&$\mathbb Z_2^2\times\mathbb Z_4^2$& Yes\\

6& $\mathbb{Z}_2^4\rtimes\mathbb{Z}_2$ & 32,27& $\mathbb{Z}_2^4$ & 16,14& 9 & [0; $2^6$]& $(3,-8)$&$\mathbb Z_2^2\times\mathbb Z_4^2$& Yes\\

6& $\mathbb{Z}_2^4\rtimes\mathbb{Z}_2$ & 32,27& $\mathbb{Z}_2^4$ & 16,14& 9 & [0; $2^6$]& $(3,-8)$&$\mathbb Z_4^3$& Yes\\

6& $\mathbb{Z}_7\times D_7$ & 98,3& $\mathbb{Z}_7^2$ & 49,2& 15 & [0; $7^3$]& $(3,-8)$&$\mathbb Z_7^2$& Yes\\

6& $\mathbb{Z}_7\times D_7$ & 98,3& $\mathbb{Z}_7^2$ & 49,2& 15 & [0; $7^3$]& $(3,-8)$&$\mathbb Z_7^2$& Yes\\

6& $\mathbb{Z}_4^2\rtimes D_4$ & 128, 734& $\mathbb{Z}_4^2\rtimes\mathbb{Z}_2^2$ &64,211& 17& [0; $2^5$]& $(3,-8)$&$\mathbb Z_2\times\mathbb Z_4^2$& Yes\\

6& $(\mathbb{Z}_2^2\times D_8)\rtimes\mathbb{Z}_2$ & 128, 750& $\mathbb{Z}_2^2\times D_8$ & 64,250& 17 & [0; $2^5$]& $(3,-8)$&$\mathbb Z_2\times\mathbb Z_4^2$& Yes\\

6& $(\mathbb{Z}_2\times D_8)\rtimes\mathbb{Z}_2^2$ & 128, 1797& $\mathbb{Z}_2^2\times D_8$ & 64,250& 17 & [0; $2^5$]& $(2,-4)^2$&$\mathbb Z_2\times\mathbb Z_4^2$& Yes\\
\hline
2&$(\mathbb{Z}_2^3\rtimes D_4)\rtimes \mathbb{Z}_2^2$ & 256, 47930&$\mathbb{Z}_2^4\rtimes D_4$ & 128, 1135& 33 & [0; $2^5$]& $(3,-8)^3$&$\mathbb Z_2^3\times\mathbb Z_4$& No\\
2&$(\mathbb{Z}_4^2\rtimes \mathbb{Z}_2^2)\rtimes \mathbb{Z}_2^2$ & 256, 45303& $\mathbb{Z}_2^4\rtimes D_4$ & 128, 1135& 33 & [0; $2^5$]& $(3,-8)^2, (2,-4)^2$&$\mathbb Z_2^3\times\mathbb Z_4$& No\\
\end{tabular}}

\end{table}

\begin{thmx}\label{thmB}
The semi-isogenous mixed surfaces $X:=(C\times C)/G$ with  $p_g(X)=q(X)=1$ and $K_X^2>0$,
 form the $35$ families 
collected in Table $\ref{q1}$. In all cases $X$ is of general type.
\end{thmx}

\begin{table}[!ht]\caption{$p_g=q=1$, $K^2>0$}\label{q1}
{\small\begin{tabular}{c|c|c|c|c|c|c|c|c|c|c}
$K_X^2$ &$G$ & $Id(G)$ & $G^0$ & $Id(G^0)$ & $g(C)$ & Type & Branch Locus $\mathcal B$ &  $H_1(X,\mathbb Z)$& $g_{alb}$ & min? \\
\hline
8& $D_{2,8,5}$ & 16,6 & $\mathbb{Z}_2\times\mathbb{Z}_4$ & 8,2 &5& [1;$2^2$]& $\emptyset$& $\mathbb Z_4\times\mathbb Z^2$& 5 & Yes\\
8& $D_{2,8,3}$ & 16,8 & $D_4$ & 8,3 &5& [1;$2^2$]& $\emptyset$& $\mathbb Z_4\times\mathbb Z^2$& 5 & Yes\\
8& $\mathbb{Z}_2^2\rtimes\mathbb{Z}_4$&16,3 & $\mathbb{Z}_2^3$ & 8,5 &5& [1;$2^2$]& $\emptyset$& $\mathbb Z_2^3\times\mathbb Z^2$& 5 &Yes\\
\hline
7& $\mathbb{Z}_3\rtimes D_4$ & 24,8& $D_6$ & 12,4& 7& [1;$2^2$]& $(2,-4)$& $\mathbb Z_2\times\mathbb Z^2$& 5 & Yes\\
7& $\mathbb{Z}_3\times D_4$ & 24,10& $\mathbb{Z}_2\times\mathbb{Z}_6$ & 12,5& 7& [1;$2^2$]& $(2,-4)$& $\mathbb Z_2\times\mathbb Z^2$& 5 & Yes\\
\hline
6& $D_4$ &8,3& $\mathbb{Z}_2^2$ & 4,2& 5 & [1;$2^4$]& $(3,-8)$&$\mathbb Z_2^3\times\mathbb Z^2$& 3 & Yes \\
6& $\mathbb{Z}_3\times S_3$ & 18,3& $\mathbb{Z}_3^2$ & 9,2& 7 & [1;$3^2$]& $(3,-8)$&$\mathbb Z_3\times\mathbb Z^2$& 4 & Yes \\
6& $\mathbb{Z}_8\rtimes\mathbb{Z}_2^2$ & 32,43& $\mathbb{Z}_2\times D_4$ & 16,11& 9 & [1;$2^2$]& $(3,-8)$&$\mathbb Z_2\times \mathbb Z_4\times\mathbb Z^2$& 3 & Yes \\
6& $\mathbb{Z}_2^2\rtimes D_4$ & 32,28& $\mathbb{Z}_2\times D_4$ & 16,11& 9 & [1;$2^2$]& $(2,-4)^2$&$\mathbb Z_2\times \mathbb Z_4\times\mathbb Z^2$& 3 & Yes \\
6& $\mathbb{Z}_2^2\rtimes D_4$ & 32,28& $\mathbb{Z}_2\times D_4$ & 16,11& 9 & [1;$2^2$]& $(3,-8)$&$\mathbb Z_2^3\times\mathbb Z^2$& 3 & Yes \\
6&$\mathbb{Z}_4^2\rtimes\mathbb{Z}_2$ & 32,11& $\mathbb{Z}_4^2$ & 16,2& 9 & [1;$2^2$]& $(3,-8)$&$\mathbb Z_2\times\mathbb Z^2$& 3 & Yes \\
6& $D_8\rtimes\mathbb{Z}_2$ & 32,42& $D_4\rtimes\mathbb{Z}_2$ & 16,13& 9 & [1;$2^2$]& $(3,-8)$&$\mathbb Z_2^2\times\mathbb Z^2$& 3 &Yes\\
6& $\mathbb{Z}_4^2\rtimes\mathbb{Z}_2$ & 32,31& $\mathbb{Z}_2^2\rtimes\mathbb{Z}_4$ & 16,3& 9 & [1;$2^2$]& $(3,-8)$&$\mathbb Z_2^2\times\mathbb Z^2$& 3 & Yes \\
6& $(\mathbb{Z}_2^2\times\mathbb{Z}_4)\rtimes\mathbb{Z}_2$ & 32,30& $\mathbb{Z}_2^2\rtimes\mathbb{Z}_4$ & 16,3& 9 & [1;$2^2$]& $(2,-4)^2$&$\mathbb Z_4\times\mathbb Z^2$& 3 & Yes \\
6& $D_{2,8,5}\rtimes\mathbb{Z}_2$ & 32,38& $\mathbb{Z}_2\times\mathbb{Z}_8$ & 16,5& 9 & [1;$2^2$]& $(2,-4)^2$&$\mathbb Z_2\times\mathbb Z^2$& 3 & Yes \\
6& $\mathbb{Z}_4\times D_4$ & 32,25& $\mathbb{Z}_2^2\times\mathbb{Z}_4$ & 16,10& 9 & [1;$2^2$]& $(2,-4)^2$&$\mathbb Z_2^2\times\mathbb Z^2$& 3 & Yes \\
6& $(\mathbb{Z}_2^2\times\mathbb{Z}_4)\rtimes \mathbb{Z}_2$ & 32,30& $\mathbb{Z}_2^2\times\mathbb{Z}_4$& 16,10& 9 & [1;$2^2$]& $(3,-8)$&$\mathbb Z_2^2\times\mathbb Z^2$& 3 & Yes\\
\hline
4& $S_3\times D_4$ & 48,38& $\mathbb{Z}_2^2\times S_3$ & 24,14& 13& [1;$2^2$]& $(2,-4), (4,-12)$& $\mathbb Z_2^2\times\mathbb Z^2$& 3 & \\
4& $D_{12}\rtimes\mathbb{Z}_2$ & 48,37& $\mathbb{Z}_4\times S_3$ & 24,5& 13& [1;$2^2$]& $(2,-4), (4,-12)$& $\mathbb Z_2\times\mathbb Z^2$& 3 & \\
\hline

2& $(\mathbb{Z}_8\rtimes\mathbb{Z}_2^2)\rtimes\mathbb{Z}_2$ & 64,153& $D_{2,8,5}\rtimes\mathbb{Z}_2$ & 32,7& 17 & [1;$2^2$]& $(3,-8), (5,-16)$&$\mathbb Z_2\times\mathbb Z^2$& 3 & No \\
2& $\mathbb{Z}_8\rtimes D_4$ & 64,150& $D_4\rtimes\mathbb{Z}_4$ & 32,9& 17 & [1;$2^2$]& $(3,-8), (5,-16)$&$\mathbb Z_2\times\mathbb Z^2$& 3 &No\\
2& $\mathbb{Z}_2^2\rtimes D_8$ & 64,147& $D_4\rtimes\mathbb{Z}_4$ & 32,9& 17 & [1;$2^2$]& $(2,-4)^2, (5,-16)$&$\mathbb Z_2\times\mathbb Z^2$& 3 &No\\
2& $(\mathbb{Z}_2\times D_8)\rtimes\mathbb{Z}_2$ & 64,128& $\mathbb{Z}_2\times D_8$ & 32,39& 17 & [1;$2^2$]& $(2,-4)^2,(3,-8)^2$&$\mathbb Z_2\times\mathbb Z^2$& 3 &No\\
2& $Q\rtimes D_4$ & 64,130& $\mathbb{Z}_2\times D_{2,8,3}$ & 32,40& 17 & [1;$2^2$]& $(3,-8)^3$&$\mathbb Z_2\times\mathbb Z^2$& 3 &No\\
2& $D_4\rtimes D_4$ & 64,134& $\mathbb{Z}_8\rtimes\mathbb{Z}_2^2$ & 32,43& 17 & [1;$2^2$]& $(3,-8)^3$&$\mathbb Z_2\times\mathbb Z^2$& 3 &No\\

2& $(\mathbb{Z}_2\times D_4)\rtimes\mathbb{Z}_2^2$ & 64,227& $\mathbb{Z}_2^3\rtimes\mathbb{Z}_4$ & 32,22& 17 & [1;$2^2$]& $(3,-8)^2,(2,-4)^2$&$\mathbb Z_2^2\times\mathbb Z^2$& 2 &No\\
2& $(\mathbb{Z}_2\times D_4)\rtimes\mathbb{Z}_2^2$ &64,227& $\mathbb{Z}_2^3\rtimes\mathbb{Z}_4$& 32,22& 17 & [1;$2^2$]& $(3,-8)^2,(2,-4)^2$&$\mathbb Z_2^2\times\mathbb Z^2$& 2 &No\\
2& $\mathbb{Z}_4\rtimes(D_4\rtimes\mathbb{Z}_2)$ & 64,228& $(\mathbb{Z}_4\rtimes \mathbb{Z}_4)\times\mathbb{Z}_2$ & 32,23& 17 & [1;$2^2$]& $(3,-8)^2,(2,-4)^2$&$\mathbb Z_2^2\times\mathbb Z^2$& 2 &No\\
2& $(\mathbb{Z}_4\times D_4)\rtimes\mathbb{Z}_2$ & 64,234& $(\mathbb{Z}_4\rtimes \mathbb{Z}_4)\times\mathbb{Z}_2$ & 32,23& 17 & [1;$2^2$]& $(3,-8)^3$&$\mathbb Z_2^2\times\mathbb Z^2$& 2 &No\\
2& $(\mathbb{Z}_4\times D_4)\rtimes\mathbb{Z}_2$ & 64,234& $\mathbb{Z}_4^2 \rtimes\mathbb{Z}_2$ & 32,24& 17 & [1;$2^2$]& $(3,-8)^2,(2,-4)^2$&$\mathbb Z_2^2\times\mathbb Z^2$& 2 &No\\
2& $(\mathbb{Z}_4\rtimes Q)\rtimes\mathbb{Z}_2$  & 64,236& $\mathbb{Z}_4^2\rtimes\mathbb{Z}_2$ &32,24& 17 & [1;$2^2$]& $(3,-8)^3$&$\mathbb Z_2^2\times\mathbb Z^2$& 2 &No\\
2& $\mathbb{Z}_4^2\rtimes\mathbb{Z}_2^2$ & 64,219& $\mathbb{Z}_4\times D_4$ &32,25& 17 & [1;$2^2$]& $(3,-8)^3$&$\mathbb Z_2^2\times\mathbb Z^2$& 2 &No\\
2& $(\mathbb{Z}_2^2\rtimes D_4)\rtimes\mathbb{Z}_2$ & 64,221& $\mathbb{Z}_4\times D_4$ & 32,25& 17 & [1;$2^2$]& $(3,-8)^3$&$\mathbb Z_2^2\times\mathbb Z^2$& 2 &No\\
2& $(\mathbb{Z}_2\times\mathbb{Z}_4)\rtimes D_4$ & 64,213& $\mathbb{Z}_4\times D_4$ &32,25& 17 & [1;$2^2$]& $(3,-8)^2,(2,-4)^2$&$\mathbb Z_2^2\times\mathbb Z^2$& 2 &No\\
2& $\mathbb{Z}_4^2\rtimes\mathbb{Z}_2^2$ & 64,206& $\mathbb{Z}_4\times D_4$ &32,25& 17 & [1;$2^2$]& $(3,-8),(2,-4)^4$&$\mathbb Z_2^2\times\mathbb Z^2$& 2 &No
\end{tabular}}
\end{table}

\begin{thmx}\label{thmC}
The semi-isogenous mixed surfaces $X:=(C\times C)/G$ with $p_g(X)=q(X)=2 $ and $K_X^2>0$,
 form the $9$ families 
collected in Table $\ref{q2}$. In all cases $X$ is of general type.
\end{thmx}

\begin{table}[!ht]
\caption{$p_g=q=2$, $K^2>0$}\label{q2}
{\small
\begin{tabular}{c|c|c|c|c|c|c|c|c|c}
$K_X^2$&$G$ & $Id(G)$ & $G^0$ & $Id(G^0)$ & $g(C)$ & Type & Branch Locus $\mathcal B$ &  $H_1(X,\mathbb Z)$&  min? \\
\hline
8&  $\mathbb{Z}_4$ &4,1&$\mathbb{Z}_2$&2,1& 3&[2;-] &$\emptyset$&$\mathbb Z_2\times\mathbb Z^4$& Yes\\
\hline
7& $ \mathbb{Z}_6$&6,2& $ \mathbb{Z}_3$&3,1& 4&[2;-] &$(2,-4)$&$\mathbb Z^4$&  Yes\\
\hline
6& $D_4$&8,3&  $\mathbb{Z}_2^2$ &4,2& 5 &[2;-]& $(3,-8)$& $\mathbb Z^4$& Yes\\
6&  $D_4$&8,3& $\mathbb{Z}_2^2$&4,2& 5 &[2;-]& $(3,-8)$&$\mathbb Z_2\times\mathbb Z^4$& Yes\\
6&  $\mathbb{Z}_2\times\mathbb{Z}_4$&8,2& $\mathbb{Z}_4$& 4,1&5&[2;-]& $(2,-4)^2$& $\mathbb Z^4$&Yes\\
\hline
4&$ D_6 $&12,4&$S_3$&6,1&7& [2;-]&$(2,-4),(4,-12)$ & $\mathbb Z^4=\pi_1(X)$& No, $K^2_{X_{min}}=5$\\
\hline
2 & $\mathbb{Z}_2\times D_4$&16,11& $ D_4$&8,3&9&[2;-]&$(2,-4)^2,(3,-8)^2$& $\mathbb Z^4=\pi_1(X)$& No, $K^2_{X_{min}}=4$\\
2 & $\mathbb{Z}_2\times D_4$&16,11& $ D_4$&8,3&9&[2;-]&$(2,-4)^2,(3,-8)^2$& $\mathbb Z^4=\pi_1(X)$& No, $K^2_{X_{min}}=4$\\
2 & $ D_4\rtimes\mathbb{Z}_2$ &16,13& $Q$&8,4&9&[2;-]&$(3,-8)^3$& $\mathbb Z^4=\pi_1(X)$& No, $K^2_{X_{min}}=4$\\
\end{tabular}}
\end{table}

\begin{remark}
By the above mentioned papers, we already have a complete classification of the surfaces of general type with $p_g=q\geq 3$.
There exists a unique family of semi-isogenous mixed surfaces  with  $p_g=q\geq 3 $ and $K^2>0$,
 it has $p_g=q= 3 $, $K^2=6$. It is the family of the symmetric products of curves of genus 3  and it
 forms a connected component of  dimension 6 of the moduli space of minimal surfaces of general type.
 \end{remark}

In Table \ref{q0}, \ref{q1}, \ref{q2}, every row corresponds to a family and 
 we use the following notation: 
columns  $Id(G)$ and $Id(G^0)$  report the MAGMA identifier of the groups $G$ and $G^0$:
 the pair $(a,b)$ denotes the $b^{th}$ group of order $a$ in the database of Small Groups.
 
In columns $G$ and $G^0$ and throughout  the paper  we denote by $\mathbb  Z_n$ the cyclic group of order $n$, by
  $S_n$ the symmetric group on $n$ letters,   by $Q$ the group of quaternions, by $D_n$ the 	dihedral group of order $2n$,
  and by	 $D_{p,q,r}$ the group $\langle x,y\mid x^p= y^q=1, xyx^{-1}=y^r\rangle$.
The groups $(258,3678)$ and $(258,3679)$  do not have a representation as semidirect product of non trivial 
groups of smaller order, so we leave the relative spots blank.

The column Type gives the type of the generating vector for  $G^0$  is a short way, e.g. $[0; 2^5]$ stands for
$(0; 2,2,2,2,2)$.
The Branch Locus $\mathcal B$ of $\eta\colon C\times C \to X$ is also given in a short way, e.g. $(3,-8)^2, (2,-4)^2$ means 
that $\mathcal B$ consists of 4 curves, two of genus 2 and  self-intersection $-4$ and two 
of genus 3 and  self-intersection $-8$.
The last column ``min?" report (if known) whether the surface $X$ is minimal or not.
In Table \ref{q1}, we also report
 the genus $g_{alb}$ of a general fibre of the Albanese map, which is a very important deformation invariant.

In Theorem \ref{thmA}, the assumption $|G^0|\leq 2000, \neq 1024 $  is a computational assumption;
the algorithm works for arbitrary values of the invariants $K^2$, $p_g$ and $q$,
 but the implemented MAGMA version has some limitations (see Remark \ref{techlim}), and it is forced to 
 skip some cases. We report the list of the ``skipped'' cases for $K^2>0$ and $p_g = q = 0$ in Table \ref{tabSkip}.

In the  Section \ref{On_min} we study the minimality of semi-isogenous mixed surfaces. 
We describe rational curves on such surfaces and, using the Hodge Index Theorem, 
we get an upper bound for the number of $(-1)$-curves lying on them. 
As an interesting byproduct we get that all semi-isogenous mixed surfaces with $\chi=1$ and $K^2\geq 6$ are minimal. 
Finally, we solve the minimality problem for all semi isogenous-mixed surfaces with $K^2>0$ and $p_g=q=2$.		
We plan to investigate the minimality in the cases $p_g=q\leq 1$ in a subsequent paper.

The semi-isogenous mixed surfaces with $K^2_X=8\chi(X)$ are those for which the action is free; 
indeed all the examples with $K^2=8$ in Tables  \ref{q0}, \ref{q1} and \ref{q2} appeared in the papers 
already cited in this Introduction.

In Table \ref{q0}, there is a surface with $K^2=6$ and $Id(G)=(32,43)$.
It realizes a new topological type of surface of general type with $p_g=0$, 
indeed its fundamental group is different from those present in literature
(see \cite{BCGP12, BCF14, inoue, kulikov}).
To the best of our knowledge the surfaces with  $K^2=6$ and 
$H_1=\mathbb Z_7^2$ or $H_1= \mathbb Z_2 \times  \mathbb Z_4^2$ 
provide the first examples of minimal regular surfaces of general type with such invariants,
and so realize at least other two new topological types.

 Also the examples in Table \ref{q1} with $K^2 = 6,7$ 
may be, to the best of our knowledge, new, although other surfaces with these invariants have been already constructed 
(see \cite{BCF14,Pol09, MP10, rit07, rit,  rit08, Rit15}).

Finally, we mention an example of a minimal surface of general type with $K^2=7$ and $p_g=q=2$.
The first example of a minimal surface of general type with these invariants appeared very recently: in \cite{Rit15} the author constructs such surface as double cover of an abelian surface.
Pignatelli and Polizzi, in the recent paper \cite{PP16}, show that our surface is different from Rito's one, proving that in our case the Albanese map is a generically finite triple cover.
 
The paper is organized as follows.
	
In Section  \ref{MQS}   we describe the mixed action of a finite group on a product of curves and we give the algebraic recipe which, using  Riemann's Existence Theorem, constructs mixed quotients.

Section \ref{SiMS} is dedicated to the study of semi-isogenous mixed surfaces: we describe both ramification and branch locus of the quotient map $\eta:C\times C\to (C\times C)/G$, and we compute the main invariants of such surfaces.

In Section \ref{Pi1_b} we study the fundamental group of a (semi-isogenous) mixed surface.

Section \ref{ALB} is devoted to the study of the Albanese map of a semi-isogenous mixed surface $S$ with irregularity $q(S)=1$.
We prove a formula to compute the genus of its general fibre.

In Section \ref{classif} we present the algorithm we used to classify semi-isogenous mixed surfaces and we prove Theorems \ref{thmA}, \ref{thmB} and \ref{thmC}.

Section \ref{On_min} is devoted to the study of the minimality of semi-isogenous mixed surfaces.


\section{On mixed quotients}\label{MQS}

In this paper we denote by $C$ a smooth projective curve of genus $g(C)$ and by $G$ a finite subgroup
of $\Aut(C)^2\rtimes \mathbb Z_2$  whose action is \textit{mixed}, i.e.~there are elements in $G$ exchanging the two
isotrivial fibrations of $C\times C$.  

\begin{definition}
The quotient surface $X:=(C\times C)/G$ is a \textit{mixed quotient}, and its 
minimal resolution of the singularities $S\to X$ is a \textit{mixed surface}.
\end{definition}

We denote by  $G^0\triangleleft G$ the index two subgroup $G\cap\Aut(C)^2$, 
i.e.~the subgroup consisting of those elements that do not exchange the factors.
The action is said to be \textit{minimal} if the group $G^0$ acts faithfully on both factors.

We have the the following description of minimal mixed actions:

\begin{theorem}[cf. {\cite[Proposition 3.16]{Cat00}}]
\label{teoazione} 
Let $G$ be a finite subgroup of $ \Aut(C)^2\rtimes \mathbb Z_2$ whose action is minimal and mixed. Fix $\tau'
\in G\setminus G^0$; it determines an element $\tau:=\tau'^2\in G^0$ and $\varphi\in\Aut(G^0)$ defined by $\varphi(h):=\tau'h\tau'^{-1}$.
Then, up to a coordinate change, $G$ acts as follows:
\begin{equation}
\label{azione}
\begin{split}
g(x,y)=(gx,\varphi(g)y)\\
\tau' g(x,y)=(\varphi(g)y,\tau gx)
\end{split}
\qquad \text{for }g\in G^0\,.
\end{equation}

 Conversely, for every finite subgroup $G^0 <\Aut(C)$ and $G$ extension of degree $2$ of $G^0$, fixed $\tau'\in G\setminus G^0$ and defined
 $\tau$ and $\varphi$ as above, (\ref{azione}) defines a minimal mixed action on $C\times C$.
\end{theorem}

\begin{remark}\ 
\begin{itemize}
\item[i)] By {\cite[Remark 3.10, Proposition 3.13]{Cat00}}, every mixed quotient $X$ may be obtained by a unique minimal mixed action; therefore we shall consider only mixed quotients provided by the corresponding minimal mixed action, as described in Theorem
\ref{teoazione}.\\
In this case we   identify  $G^0<\Aut(C) \times \Aut(C)$ with its projection onto the first factor.

\item[ii)] The \textit{quotient map} $\eta\colon(C\times C) \to X$ can be factorized
in the natural way:
\[   C\times C\stackrel{\sigma}{\longrightarrow}Y:=(C\times C)/G^0\stackrel{\pi}{\longrightarrow}X:=(C\times C)/G\,,
\]
where
$\pi$ is the double covering determined by the involution $\iota\colon Y\to Y$ induced by the $G$-action on $Y$,
namely $\iota [(x,y)]= [(y, \tau x)]$.
\end{itemize}
\end{remark}

The description of mixed quotients is accomplished through the theory 
  of Galois coverings between projective curves.

\begin{definition}\label{gv}
Given integers  $g'\geq 0$, $m_1, \ldots, m_r > 1$  the 
\textit{orbifold surface group of type $(g';m_1, \ldots, m_r)$} is defined as:
\[
\mathbb T(g';m_1,\ldots ,m_r):=
\langle a_1,b_1,\ldots, a_{g'},b_{g'},c_1, \ldots, c_r \mid
 c_1^{m_1}, \ldots, c_r^{m_r},\prod_{i=1}^{g'} [a_i,b_i]\cdot c_1 \cdots c_r\rangle\,.
\]
Given  a finite group $H$, a \textit{generating vector} for $H$ of type $(g';m_1,\ldots ,m_r)$ is a $(2g'+r)$-tuple of
elements of $H$:
\[V:=(d_1,e_1,\ldots, d_{g'},e_{g'};h_1, \ldots, h_r)\]
such that $V$ generates $H$, $\prod_{i=1}^{g'}[d_i,e_i]\cdot h_1\cdot h_2\cdots h_r=1$ and $\ord(h_i)=m_i$.
\end{definition}

To give a generating vector of type $(g';m_1,\ldots, m_r)$ for a finite group $H$ is equivalent to
give an \textit{appropriate orbifold homomorphism} 
\[\psi \colon \mathbb T (g';m_1,\ldots, m_r)\longrightarrow H\,, \] 
i.e.~a surjective homomorphism $\psi$  such that $\psi(c_i)$ has order $m_i$.

By Riemann's Existence Theorem (cf. \cite[Section III.3, III.4]{Mir}),
 any curve $C$ of genus $g:=g(C)$ together with an action of
a finite group $H$ on it, such that $C/H$ is a curve $C'$ of genus $g':=g(C')$, is determined (modulo automorphisms)
by the following data:
\begin{enumerate}
	\item the branch set $\{p_1, \ldots, p_r\}\subset C'$;
	\item loops $\alpha_1,\ldots,\alpha_{g'},\beta_1,\ldots,\beta_{g'},\gamma_1,\ldots,\gamma_r\in \pi_1(C'\setminus\{p_1, \ldots, p_r\})$, where $\{\alpha_i, \beta_i\}_i$ generates $\pi_1(C')$, each $\gamma_i$ is a simple geometric loop around $p_i$ and 
	 $\prod_{i=1}^{g'}[\alpha_i,\beta_i]\cdot\gamma_1\cdots\gamma_r=1\,;$
\item a generating vector $V:=(d_1,e_1,\ldots, d_g,e_g;h_1, \ldots, h_r)$ for $H$ of type $(g';m_1,\ldots ,m_r)$
	 	such that \textit{Hurwitz's formula} holds:
\[ 	2g-2=|H|\bigg(2g'-2+\sum_{i=1}^r\frac{m_i-1}{m_i}\bigg)\,.\]
	\end{enumerate}

Moreover,  the \textit{stabilizer set of $V$}, defined as 
\[\Sigma_V:= \bigcup_{g\in H}\bigcup_{j\in \mathbb Z}\bigcup_{i=1}^r \{ g\cdot h_i^j\cdot g^{-1}\}\,,\]
coincides with the subset of $H$ consisting of the automorphisms of $C$ having some fixed point.

\begin{remark}\label{algdata} 
A mixed  quotient $ X=(C\times C)/G$ determines  a finite group $G$,  an index $2$ subgroup
$G^0$, a curve  $C'= C/G^0$, 
 a set of points $\{p_1,\ldots, p_r\}\subset C'$, 
and, for every choice of $\alpha_i, \beta_j, \gamma_k \in \pi_1(C'\setminus\{p_1, \ldots, p_r\})$ as in (2),  a  generating vector $V$ for $G^0$.

Conversely, the following algebraic data:
\begin{itemize}
 	 \item a finite group $G^0$;          
            \item a curve  $C'$;
	  \item points $p_1,\ldots, p_r\in C'$, and $\alpha_i, \beta_j, \gamma_k \in \pi_1(C'\setminus\{p_1, \ldots, p_r\})$ as in (2);
             \item integers $m_1,\ldots,m_r >1$;
\item a  generating vector $V$ for $G^0$ of type $(g(C');m_1,\ldots,m_r)$;
\item a degree $2$ extension $G$ of $G^0$;
\end{itemize}

\noindent 
give a uniquely determined mixed  quotient. Indeed
by Riemann's  Existence Theorem the first $5$ data give the Galois covering $c\colon C \rightarrow C/G^0\cong C'$ branched over $\{p_1, \ldots, p_r \}$. 
The last datum determines, by Theorem \ref{teoazione}, a mixed action on $C\times C$.
\end{remark}            

Let $G$ be a finite group whose action on $C\times C$ is mixed.
Let $\Sigma$ be the subset of  $G^0$ consisting of the automorphisms of $C$ having some fixed point 
and  let $\Fix(g)\subset C\times C$ be the fixed locus of $g\in G$.

\begin{lemma}\label{fix}
Let $G$ be a finite group whose action on $C\times C$ is mixed. The following hold:
 \begin{itemize}
  \item[i)] let $g\in G^0$, then $\Fix(g)\neq\emptyset$ if and only if $g\in\Sigma\cap\varphi(\Sigma)$;
  \item[ii)] let $g\in G\setminus G^0$, then $\Fix(g)\neq\emptyset $ if and only if $ g^2\in\Sigma$.
 \end{itemize}
\end{lemma}

\begin{proof} i) Let $(x,y)\in C\times C$, then $g(x,y)=(gx,\varphi(g)y)=(x,y)$ if and only if
  $g\in\Sigma\cap\varphi^{-1}(\Sigma)$.
   We observe that $\Sigma$ is $\Inn(G^0)$-invariant and  $\varphi^2\in \Inn(G^0)$; 
   therefore $\varphi^{-1}(\Sigma)=\varphi(\Sigma)$.
  
ii) There exists a unique $h\in G^0$ such that $g=\tau'h$. Let $(x,y)\in C\times C$, then  
\[
 \tau'h(x,y)=(\varphi(h)y,\tau hx)=(x,y)\Leftrightarrow 
 \left\{\begin{array}{l}
    y=\tau h\cdot x\\
    x=\varphi(h)\cdot  y 
    \end{array}\right.
    \Leftrightarrow \left\{\begin{array}{l}
    y=\tau h\cdot x \\
    x= \varphi(h)\tau h \cdot x 
    \end{array}\right.
\]
So $\Fix(g)\neq\emptyset$ if and only if $\varphi(h)\tau h\in\Sigma$. On the other hand,  $\varphi(h)\tau h=\tau'h\tau'^{-1}\tau h=(\tau' h)^2=g^2$, by Theorem \ref{teoazione}.
\end{proof}

\begin{remark}\label{stab}
We recall that for each non trivial element $g\in G^0$ the set $\Fix(g) $ is finite, 
 because $g$ fixes finitely many points on $C$;
 the $G^0$-orbits of points on $C\times C$ with non trivial stabilizer
correspond to the singular points on $Y$, which are \textit{cyclic quotient singularities} (see \cite{BP15,MP10,Pol10}).\\
In particular the map $\sigma\colon C\times C \to Y=(C\times C)/G^0$ is \textit{quasi-\'etale}; this means that 
the branch locus has codimension at least 2.

According to \cite[Theorem 3.7]{Frap13}, the quotient map $\eta \colon C\times C \to X=(C\times C)/G$  is quasi-\'etale
 if and only if the short exact sequence
\begin{equation}
\label{equation:split}
 1\longrightarrow G^0\longrightarrow G\longrightarrow \mathbb{Z}_2\longrightarrow 1
\end{equation}
does not split, or  in other words  there are no elements of order 2 in $G\setminus G^0$.
\end{remark}

This motivates the following definition.

\begin{definition} 
Let $G$ be a finite group whose action on $C\times C$ is mixed.
We define the set
 \[
  O_2:=\{g\in G\setminus G^0: g^2=1\}\,.
 \]
  For each $g\in O_2$ we define $ R_{g}:=\Fix(g)$. 
  \end{definition}
  
  Note that each $R_g$ is a smooth irreducible curve isomorphic to $C$: $R_g=\{(x,(\tau' g) \cdot x) : x\in C\}$ (cf. proof of Lemma \ref{fix}). It is a ramification curve,  and
     the next statement shows that there are no further ramification curves.

\begin{proposition}
Let $G$ be a finite group whose  action on $C\times C$ is mixed.
Let $D$ be an irreducible curve contained in the ramification locus of the map $\eta\colon C\times C\to X$.

Then  there exists $g \in O_2$  such that $D=R_g$.
\end{proposition}

\begin{proof} 
Let $P$  be  the finite set of points fixed by a non trivial element of $G^0$.
Each point in $D\setminus P$ has stabilizer of order 2 generated by an element
in $G\setminus G^0$, otherwise the point would be stabilized by a non trivial element of $G^0$;
thus each point in $D\setminus P$ belongs to one of the curves $R_g$'s. \\
Noting that if $D\neq R_g$ then $D\cap R_g$ is a finite set, we are done.
\end{proof}

\begin{proposition}
Let $X:=(C\times C)/G$ be a mixed quotient and 
$\pi\colon Y:=(C\times C)/G^0 \to X$ be the double covering determined by the involution $\iota$ induced on $Y$
by the $G$-action. 

Then $\Sing(X)\subseteq\pi(\Sing(Y))$.
\end{proposition}

\begin{proof}
Let $u:=\sigma(x,y) \in Y$ be a smooth point, and let $z:=\iota(u)\in Y$. 
By Remark \ref{stab}, $\Stab_{G^0}(x,y)=\Stab_G(x,y)\cap G^0=\{1\}$. 
 If $u\neq z$, then  $\pi(u)$ is obviously a smooth point.
 If $u= z$, then   $\Stab_G(x,y) =\langle g\rangle\cong \mathbb Z_2$, for $g\in O_2$, and  $\dim\Fix(g) =1$,
 hence  $\Stab_G(x,y)$ is generated by a quasi-reflection. 
 By the  Chevalley-Shephard-Todd Theorem (cf. \cite[Theorem 5.1]{ST54}),
 the point $\pi(u)$ is smooth.
\end{proof}

Since the singular points on $Y$ correspond to the $G^0$-orbits of points on $C\times C$ with non trivial stabilizer, we have:

\begin{corollary}
 Let $X:=(C\times C)/G$ be a mixed quotient and suppose the action of $G^0$ on $C\times C$ to be free. Then $X$ is smooth.
\end{corollary}

\section{Semi-isogenous Mixed Surfaces}\label{SiMS}

\begin{definition}
Let $X:=(C\times C)/G$ be a mixed quotient and let $Y:=(C\times C)/G^0$.
If $Y$ is a surface isogenous to a product, i.e.~$G^0$ acts freely, 
then $X$ is a \textit{semi-isogenous mixed surface}.
\end{definition}

\begin{remark}\label{disj}
To construct a semi-isogenous mixed surface one has to give the 
same algebraic data as in Remark \ref{algdata} and require that  $G^0$ acts freely. 
This is equivalent to imposing that the degree 2 extension
$G$ of $G^0$ satisfies the following condition:
let $\varphi\in\Aut(G^0)$ as in Theorem \ref{teoazione},
then the stabilizer sets $\Sigma_V$ and $\Sigma_{\varphi(V)}(=\varphi(\Sigma_V))$ are \textit{disjoint}, 
i.e.~$\Sigma_V\cap\Sigma_{\varphi(V)}=\{1\}$.
\end{remark}

\begin{proposition}\label{ram_loc}
Let $X:=(C\times C)/G$ be a semi-isogenous mixed surface.
Then the ramification locus of the quotient map $\eta \colon C\times C\to X$ is the disjoint union
\[\bigsqcup_{g \in O_2} R_g  \,.\]
 \end{proposition}

\begin{proof}
Let  $(x,y)\in C\times C$ be a point with non trivial stabilizer.
Since $G^0$ acts freely, there exists $g\in O_2 \cap\Stab_G(x,y)$, i.e.~$(x,y)\in R_{g}$. 

Let $g$ and $ h$ be two  elements of $O_2$ and assume that  $(x,y)\in R_g \cap R_h$. Then
$g^{-1}h (x,y)=(x,y)$, but $g^{-1}h \in G^0$ and fixes a point, whence $g=h$.
\end{proof}

\begin{lemma}
\label{ram_conj}
 Let $X:=(C\times C)/G$ be a semi-isogenous mixed surface and let $g, h\in O_2$. 
 
  Then   $h=\gamma g\gamma^{-1}$   for $\gamma\in G$  if and only if $\gamma R_g=R_h$.
 
 In particular,  $R_g$ is $\gamma$-invariant if and only if $\gamma$ belongs to the centralizer of $g$: $ Z(g)$.
\end{lemma}
\begin{proof}
For any $\gamma \in G$, the curve $\gamma R_{g}$ is fixed pointwise by $\gamma g\gamma^{-1}=h$,
 hence   $\gamma R_{g}=R_{h}$.
 
Conversely, if $\gamma R_{g}= R_{h}$, then $\gamma g\gamma^{-1}$ fixes $ R_{h}$ pointwise, i.e.~$\gamma g\gamma^{-1}=h$.
 \end{proof}

Let $G$ be a finite group whose action on $C\times C$ is mixed.
We denote  by $\Cl(g)$ the conjugacy class of $g\in G$
and we define $   \Cl(O_2):=\{\Cl(g) : g\in O_2\}$. 
 
 \begin{proposition}\label{br_loc}
  Let $X:=(C\times C)/G$ be a semi-isogenous mixed surface.

 Then the branch locus $\mathcal B$  of the quotient map $\eta \colon C\times C\to X$ is the disjoint union 
 $\mathcal B=B_{g_1}\sqcup\cdots \sqcup B_{g_t}$, where $t:=|\Cl(O_2)|$,
 $\{g_1,\ldots, g_t\}$ is a set of representative of the conjugacy classes in $\Cl(O_2)$ and  $B_{g_i}:=\eta(R_{g_i})$.

Moreover, for each $g\in O_2$, the  map $\eta|_{ R_{g}}: R_{g}\to B_g=:\eta( R_{g})$ is an unbranched covering of degree $|Z(g)|/2$ and  
\begin{equation}   \label{genus_branch}
   g(B_g)=\frac{2(g(C)-1)}{|Z(g)|}+1.
  \end{equation}
 \end{proposition}

 \begin{proof}
 The first claim is a direct consequence of  Proposition \ref{ram_loc} and Lemma \ref{ram_conj}.
  
For each $g\in O_2$, the  map $\eta|_{ R_{g}}: R_{g}\to B_g $ is  unbranched since $G^0$ acts freely by Lemma \ref{ram_conj}
 and its degree is
  $\deg(\eta|_{ R_{g}})=|Z(g)|/|\langle g\rangle|=|Z(g)|/2$.
Since $R_g\cong C$,  the equation (\ref{genus_branch}) follows now from Hurwitz's formula.
 \end{proof}

 \subsection{The invariants of a semi-isogenous mixed surface}

\begin{proposition}
\label{irreg}
 Let $X:=(C\times C)/G$ be a semi-isogenous mixed surface, then
 $q(X)$ equals the genus of $C':=C/G^0$.
\end{proposition}

\begin{proof}
 Arguing as in {\cite[Proposition 3.15]{Cat00}}:
\[\begin{array}{rcl}
  H^0(\Omega^1_X)	&=& (H^0(\Omega^1_{C\times C}))^G=(H^0(\Omega^1_C)\oplus H^0(\Omega^1_C))^G\\
			&=&(H^0(\Omega^1_C)^{G^0}\oplus H^0(\Omega^1_C)^{G^0}))^{G/G^0}\\
			&=&(H^0(\Omega^1_{C'})\oplus H^0(\Omega^1_{C'}))^{G/G^0}.
\end{array}
 \]
Where first and last equalities hold for {\cite[pp.~78-79]{Beau83}}. Since $X$ is a mixed quotient, 
$G/G^0\cong\mathbb{Z}_2$ exchanges the last two summands, hence $q(X)=h^0(\Omega^1_X)=h^0(\Omega^1_{C'})=g(C')$.
\end{proof}

Let $X:=(C\times C)/G$ be a semi-isogenous mixed surface.
 We denote by $\mathcal B:=\{\eta( R_{g}):g\in O_2\}=B_1\cup\ldots\cup B_t$ the branch 
locus of the quotient map $\eta\colon C\times C \to X$, and we define the  integer 
\[\delta(\mathcal{B}):=\sum_{j=1}^t\, (g(B_j)-1)\,.\]

 \begin{remark}
By Proposition \ref{br_loc}, the branch curves are pairwise disjoint, hence $\delta(\mathcal{B})=p_a(\mathcal{B})-1$, where $p_a(\mathcal{B})$ denotes the arithmetic genus of $\mathcal{B}$.
 \end{remark}
 
 \begin{proposition}
  \label{eultop_k2}
  Let $X:=(C\times C)/G$ be a semi-isogenous mixed surface, then
  \begin{equation}\label{eultop}
   e(X)=\frac{2(g(C)-1)}{|G|}\cdot(2(g(C)-1)-|O_2|)= \frac{4(g(C)-1)^2}{|G|}-\delta(\mathcal B),
  \end{equation}
 and
 \begin{equation} \label{k2}
 K^2_X=\frac{2(g(C)-1)}{|G|}\cdot(4(g(C)-1)-5\cdot|O_2|) =\frac{8(g(C)-1)^2}{|G|}-5\delta(\mathcal B)\,.
 \end{equation}
 \end{proposition}

 \begin{proof}
 For $i=1,\dots,t$, let $g_i\in O_2$ such that $\eta( R_{g_i})=B_i$.
By Proposition \ref{br_loc}
\[
 e(B_i)=\frac{-4(g(C)-1)}{|Z(g_i)|}=-4\frac{N_i}{|G|}(g(C)-1),
\]
where $N_i=\frac{|G|}{|Z(g_i)|}$ is the cardinality of the conjugacy class of $g_i$.
Note that $\sum_{i=1}^t{N_i}=|O_2|$.
Since the ramification locus $R:=\eta^{-1}(\mathcal B)$ is the disjoint 
union of $|O_2|$ smooth curves isomorphic to $C$,  it holds
  \[
 e(C\times C\setminus R)	=e(C\times C)-e(R)=4(g(C)-1)^2+ 2|O_2|(g(C)-1)\,.
\]
It follows: 
\[\begin{array}{rcl}
 e(X)	&=&e(X\setminus \mathcal B)+e(\mathcal B)=\dfrac{e(C\times C\setminus R)}{|G|}+\sum_{i=1}^t{e(B_i)}\\[10pt]
	&=&\dfrac{4(g(C)-1)^2}{|G|}+\dfrac{2|O_2|}{|G|}(g(C)-1)-\dfrac{4(g(C)-1)}{|G|}\sum_{i=1}^t{N_i}\\[10pt]
		&=&\dfrac{2(g(C)-1)}{|G|}\cdot (2(g(C)-1) +|O_2|- 2|O_2| )\,.
\end{array}
\]
The second equality in (\ref{eultop}) follows now from $\delta(\mathcal{B})=\dfrac{2(g(C)-1) |O_2|}{|G|}$.

The map $\sigma \colon C\times C\to Y:=(C\times C)/G^0$ is an unramified covering of degree $|G^0|=|G|/2$ and 
the canonical divisor  $K_{C\times C}$ is numerically equivalent to
$ (2g(C)-2)F_1+(2g(C)-2)F_2$,
where $F_1,F_2$ denote a general fiber of the projection  onto the first and onto the
second coordinate, respectively. Then
\[
 K_Y^2=\frac{K_{C\times C}^2}{\deg(\sigma)}=\frac{16(g(C)-1)^2}{|G|}.
\]
On the other side, $\pi\colon Y\to X$ is a double covering branched over $\mathcal{B}$; by the ramification formula, 
the canonical divisor $K_Y$ is numerically equivalent to $\pi^*(K_X+\mathcal B/ 2)$ and we get
\[
K_Y^2=2\left(K_X^2+\frac{\mathcal B^2}{4}+K_X.\mathcal B\right).
\]
By Proposition \ref{br_loc}, $\mathcal B^2=B_1^2+\dots+B_t^2$. 
For all $i=1,\dots,t$, $\eta^*(B_i)=2R_{i,1}+\dots +2R_{i,N_i}$; these curves are pairwise disjoint and of genus $g(C)$
 by Proposition \ref{ram_loc} .
 Applying  the adjunction formula to $R_{i,j}\subset C\times C$, we get:
\[\begin{array}{rcl}
(R_{i,j})^2&=& 2(g(C)-1)- K_{C\times C}.R_{i,j}= 2(g(C)-1)- [(2g(C)-2)F_1.R_{i,j}+(2g(C)-2)F_2.R_{i,j}]\\[10pt]
&=& -2(g(C)-1)\,.
\end{array}
\]
According to the projection formula it holds:
\[
 B_i^2=\frac{1}{|G|}\eta^*(B_i).\eta^*(B_i)=\frac{4}{|G|}\sum_{j=1}^{N_i}{(R_{i,j})^2} =-\frac{8N_i(g(C)-1)}{|G|}\,.
\]
Using again the adjunction formula, one easily gets
\[
 K_X.B_i=2(g(B_i)-1) - B_i^2=\frac{4N_i(g(C)-1)}{|G|}+\frac{8N_i(g(C)-1)}{|G|}.
\]
Finally
\[\begin{array}{rclcl}
 K_X^2	&=& \dfrac{K_Y^2}{2}-K_X.\mathcal B-\dfrac{\mathcal B^2}{4}
		&=& \dfrac{8(g(C)-1)^2}{|G|}-\dfrac{12(g(C)-1)}{|G|}\sum_{i=1}^t{N_i}+\dfrac{2(g(C)-1)}{|G|}\sum_{i=1}^t{N_i}\\[10pt]
		& & &=&\dfrac{2(g(C)-1)}{|G|}\cdot (4(g(C)-1)- 6|O_2|+|O_2|).
\end{array}
\]
\end{proof}

By Proposition \ref{eultop_k2} and Noether's formula, we immediately get the following.

\begin{corollary}
 \label{cor_chi}
 Let $X:=(C\times C)/G$ be a semi-isogenous mixed surface. Then 
\begin{equation}
 \label{eq_chi}
 \chi(\mathcal O_X)=\frac{(g(C)-1)}{|G|}(g(C)-1-|O_2|)=\frac{(g(C)-1)^2}{|G|}-\frac 12\delta(\mathcal B).
\end{equation}
 \end{corollary}

 \begin{remark}\label{num_br}
i) Combining  Proposition \ref{eultop_k2} and Corollary \ref{cor_chi} we  get
\begin{equation}\label{8chi}
   8\chi(\mathcal O_X)-K^2_X=\frac{(g(C)-1)}{|G^0|}\cdot |O_2|= \delta(\mathcal B)\,.
   \end{equation}

ii) It follows from the proof of Proposition \ref{eultop_k2}  that
\[
 K_X.\mathcal{B}=6\cdot \delta(\mathcal{B})\,, \qquad \mathcal{B}^2=-4 \cdot \delta(\mathcal{B}) \qquad \mbox{and} \qquad 
 10\chi(\mathcal{ O}_X)-K^2_X=\chi(\mathcal{ O}_Y)\,. 
\]
 \end{remark}
\subsection{The classification for $g(C)=0,1$ } \label{class01}

\begin{itemize}
\item[(0)] Let $X:=(C\times C)/G$ be a semi-isogenous mixed surface, where $C\cong\mathbb{ P}^1$.
Since each automorphism of $\mathbb P^1$ has non empty fixed locus, the only possibility
is that $G^0$ is trivial and $G\cong \mathbb Z_2$ is generated by the involution that exchanges the
two factors. Therefore, $X$ is the symmetric product $(\mathbb P^1)^{(2)} \cong \mathbb{P}^2$.

\item[(1)] Let $X:=(C\times C)/G$ be a semi-isogenous mixed surface, where $C$ is a curve of genus $g(C)=1$.
In this case $X$ is a surface with invariants $K^2_X=e(X)=\chi(X)=0$ and $q(X)=1$, $p_g(X)=0$, while
the double covering $Y:=(C\times C)/G^0$  has $q(Y)=2$ and Kodaira dimension  $\kappa(Y)=\kappa(C\times C) = 0$,
  hence $Y$ is an abelian surface by the Enriques-Kodaira classification. 
  Looking at the fixed locus of the involution $\iota\colon Y\to Y$, we distinguish two cases.
 \begin{itemize}
 \item[a)]The fixed locus is empty, then $\kappa(X)=0$ and $X$ is a bi-elliptic surface.
 \item[b)] The fixed locus  consists of some elliptic curves, then  $\kappa(X)=-\infty$ 
 (see \cite[Lemma 2.6]{Kat87}), whence $X$ is a ruled surface of genus 1.
\end{itemize}
\end{itemize}

\section{The fundamental group}\label{Pi1_b}

In this section we show how to compute the fundamental group of a (semi-isogenous) mixed surface;
 we follow the strategy described in \cite{Frap13}.

Let $X:=(C\times C)/G$ be a mixed quotient 
and  $\psi\colon \mathbb T(g';m_1,\ldots ,m_r)\rightarrow G^0$ be the appropriate orbifold homomorphism
associated to the $G^0$-covering  $C\rightarrow C/G^0$ (see Remark \ref{algdata}).
As already remarked in \cite{BCGP12}, the kernel of $\psi$ is isomorphic to  the fundamental group $\pi_1(C)$
and the action of $\pi_1(C)$ on the universal covering $\Delta$ of $C$ extends to a faithful discontinuous
action of $\mathbb T:= \mathbb T(g';m_1,\ldots, m_r)$.
The covering map $u\colon \Delta \rightarrow C$ is $\psi$-equivariant, 
and $ \Delta/\mathbb T \cong C/G^0 $.

Fix $\tau'\in G\setminus G^0$; let $\tau=\tau'^2\in G^0$ and
$\varphi\in \Aut(G^0)$ defined 
by  $\varphi(h):=\tau'h \tau'^{-1}$. \\
Let $\mathbb H:=\{(t_1,t_2)\in \mathbb 	T \times \mathbb T \mid \psi(t_1)=\varphi^{-1}(\psi(t_2))\}
<\Aut(\Delta\times \Delta)$.
Since $\psi$ is surjective and $\varphi(\tau)=\tau$, there exists $t\in \mathbb T$ such that
$\tilde\tau:=(t,t)\in \mathbb H$.
We define the automorphism
\[
\begin{array}{rcc}
\tilde \tau':\Delta\times \Delta& \longrightarrow& \Delta\times \Delta\,,\\
(x,y)&\longmapsto& (y,t\cdot x)
\end{array}\]
which satisfies $(\tilde\tau')^2=\tilde\tau$.
We  also define $\tilde\varphi \colon \mathbb H\rightarrow \mathbb H $ as the conjugation by $\tilde\tau'$:
$\tilde\varphi(t_1,t_2)=(t_2,t\cdot t_1 \cdot t^{-1})$.
Let $\mathbb H=\langle gen(\mathbb H) \mid rel(\mathbb H)\rangle$ be a finite presentation of $\mathbb H$,
we define
\[REL:=\{\tilde\varphi(h)\tilde\tau'h^{-1}\tilde\tau'^{-1} \mid h\in gen(\mathbb H)\}\,.\]

Let $\mathbb{G}$ be the subgroup of $\Aut(\Delta\times\Delta)$ generated by $\mathbb{H}$ and $\tilde{\tau}'$; 
a finite presentation of $\mathbb G$ is:
\[\mathbb G:=\langle gen(\mathbb H), \tilde\tau' \mid rel(\mathbb H), (\tilde\tau')^2\tilde\tau^{-1}, REL\rangle\,.\]
We note that $\mathbb H$ is an index 2 subgroup of $\mathbb G$,
and we have a natural faithful action  $\mathbb G<\Aut(\Delta\times \Delta)$:
\[
\begin{array}{rcl}
(h_1,h_2)\cdot (x,y) &=&(h_1\cdot x, h_2 \cdot y)\\
\tilde\tau'(h_1,h_2) \cdot (x,y)& = &(h_2\cdot y, (t\cdot h_1)\cdot x)
\end{array}
\qquad \mbox{ for } (h_1,h_2)\in \mathbb H\,.
\]

\begin{theorem}\label{ThmPi1}
Let $X:=(C\times C)/G$ be a semi-isogenous mixed surface. Then 
\[\pi_1(X)\cong \frac{\mathbb G}{\mathbb G'}\,,\] 
where $\mathbb G'$ is the normal subgroup of $\mathbb G$ generated by those elements 
which have fixed points.
\end{theorem}

\begin{proof}
Let $\vartheta\colon\mathbb G\rightarrow G$ be the  surjective  group morphism defined by:
\[
\begin{array}{rcl}
\vartheta(h_1,h_2)&=&\psi(h_1)\\
\vartheta(\tilde\tau'(h_1,h_2))&= &\tau'\psi(h_1)
\end{array}
\qquad \mbox{ for } (h_1,h_2)\in \mathbb H\,.
\]

Let $\mathcal U:=(u,u)\colon \Delta\times \Delta\rightarrow C\times C$.
It is straightforward to prove that $\mathcal U$ is $\vartheta$-equivariant 
and so
\[\frac{\Delta\times \Delta}{\mathbb G} \cong \frac{C\times C}{G}\,.\]
Since the $\mathbb G$-action on $\Delta\times \Delta$ is discontinuous (see \cite[Lemma 5.3]{Frap13}), 
the main theorem in \cite{Arm68} applies and we get:
\[\pi_1\bigg(\frac{C\times C}{G}\bigg)\cong\pi_1\bigg(\frac{\Delta\times \Delta}{\mathbb G}\bigg)\cong \frac{\mathbb G}{\mathbb G'}\,.\]

\end{proof}

\begin{remark}
Note that the above proof works for arbitrary mixed quotients $X:=(C\times C)/G$. If $f\colon S\to X$ is a mixed surface,
then  \[\pi_1(S)=\pi_1(X)= \dfrac{\mathbb{G}}{\mathbb{G}'}\,.\]
Indeed, by construction, $X$ is normal and has only quotient singularities. 
According to \cite[Theorem 7.8]{Kollar93}, the natural morphism
\[f_*\colon \pi_1(S) \longrightarrow \pi_1(X)\]
induced by the resolution is an isomorphism.
\end{remark}

The next statement explains how to find a finite set of generators of $\mathbb G'$,
for a semi-isogenous mixed surface $X:=(C\times C)/G$.

\begin{proposition}\label{fingen}
 Let $X:=(C\times C)/G$ be a semi-isogenous mixed surface.
  Then  $\mathbb G'$ is normally generated by the finite set   $\mathcal N$ defined as follows:
for each element $h\in O_2$, we choose an element 
$ h_1 \in \psi^{-1}(h)$ and we include in $\mathcal{N}$ the element
$\tilde\tau'(h_1,(t\cdot h_1)^{-1})\in \mathbb{G}$.

\end{proposition}

\begin{proof}

Let $(h_1,h_2)\in \mathbb H$ and assume that it fixes the point $(x,y)\in \Delta\times\Delta$, 
i.e.~$(h_1,h_2)\cdot (x,y)=(x,y)$. The map $\mathcal U $ is  $\vartheta$-equivariant, hence 
 $\vartheta(h_1,h_2) \in G^0$ fixes the point $\mathcal U(x,y)\in C\times C$, but
 $G^0$ acts freely and so $(h_1,h_2)\in \ker \varphi=\pi_1(C\times C)$.
 Now we note that $\pi_1(C\times C)$ acts freely on $\Delta\times\Delta$, hence $(h_1,h_2)$ is trivial;
 in other words, $\mathbb H$ acts freely on $\Delta\times\Delta$.

Let $g:=\tilde\tau'(h_1,h_2)\in \mathbb G\setminus \mathbb{H}$ and assume that it fixes the point $(x,y)\in \Delta\times\Delta$.
We note that $g^2\in \mathbb H$ and fixes $(x,y)$, hence  $g$ has order 2 (i.e.~$ h_2= (t\cdot h_1)^{-1}$).
Conversely each element  $\tilde\tau'(h_1, (t\cdot h_1)^{-1}) \in\mathbb G\setminus \mathbb{H}$ 
fixes point-wise the curve $\{(x, (t\cdot h_1)\cdot x)\mid x \in \Delta\}$.

Each element $g\in\mathbb G\setminus \mathbb{H}$ of order two maps, via $\vartheta$, to an element of order 2 in $G\setminus G^0$.
If two such elements map via $\vartheta$ to the same element, then they are conjugated in $\mathbb G$.
Indeed, let $\tilde\tau'(h_1,(t\cdot h_1)^{-1}), \tilde\tau'(s_1,(t\cdot s_1)^{-1})$ such that $\psi(h_1)=\psi(s_1)$, then
there exists $k\in\ker\psi=\pi_1(C)$ such that $h_1=s_1 \cdot k$. 
We have the following equalities:
\[\begin{array}{rcl}
\tilde\tau'(h_1,(t\cdot h_1)^{-1})&=& \tilde\tau'(s_1 \cdot k, k^{-1} s_1^{-1} t^{-1})
=\tilde\tau' (1,k^{-1} )\cdot(s_1, (t\cdot s_1)^{-1})\cdot (k,1)\\
&=&\tilde \varphi(1,k^{-1})\cdot \tilde\tau' (s_1, (t \cdot s_1)^{-1}) \cdot(k,1)\\
&=&(k,1)^{-1}\cdot\tilde\tau' (s_1, (t\cdot s_1)^{-1}) \cdot(k,1)\
\end{array}
\]
and we are done, because $(k,1)\in \mathbb{H}$.
\end{proof}
 
 \section{The Albanese fibre of a  semi-isogenous mixed surface with irregularity 1}\label{ALB}

 The Albanese map of a surface $X$ with irregularity $q(X)=1$ is a fibration
 onto the elliptic curve $\Alb(X)$ and the genus $g_{alb}$ of the general Albanese fibre is an important deformation
 invariant. In this section we explain how to compute $g_{alb}$ for a semi-isogenous mixed surface.
 
 The  argument is analogous to the one  in \cite[Section 3]{FP15}   and we refer to it for further details.
 
 Let $X:=(C\times C)/G$ be  a semi-isogenous mixed surface with $q(X)=1$. By
 Proposition \ref{irreg}, $E:=C/G^0$ is an elliptic curve and the Galois covering $c\colon C\to E$ has branch 
 locus $B:=\{p_1,\dots,p_r\}$. Up to translation, we may assume that the neutral element 
 $0\in E$ is not in $B$ and that $-p_i\not\in B$ for  each $i\in\{1,\dots,r\}$. 
We have the following commutative diagram:
\begin{equation}\label{diag}
 \xymatrix{
& C \times C \ar[d]_{\eta} \ar[r]^Q & E\times E \ar[d]^\epsilon \\
& X \ar[r] \ar[d]_f \ar[dr]^{\alpha} 	& E^{(2)}\ar[d]^{\tilde{\alpha}}\\
&\Alb(X) \ar[r]^{\psi}	& E}
\end{equation}
 where $\tilde{\alpha}$ is the Abel-Jacobi map and $Q:= c \times c$. By the properties of the Albanese torus (see {\cite[Proposition I.13.9]{BHPV}}),
 the Stein factorization of $\alpha$ is given by the Albanese map $f \colon  X\to\Alb(X)$ and a (unique)
 homomorphism $\psi\colon \Alb(X)\to E$.
 
  Let $E':=\epsilon^*(\tilde{\alpha}^*(0))=\{(u,-u):u\in E\}$, consider $F^*:=Q^*(E')$ and
 $F:=\alpha^*(0)$. Note that $\eta(F^*)=F$ and that  $F^*$ and $F$ are smooth, because $-p_i\not\in B$.
 Let us define the points $q_i:=(p_i,-p_i)$ and $q_i':=(-p_i,p_i)$ of $E'$ and set $B':=\{q_i,q_i'\}$; we note that $0':=(0,0)\in E'\setminus B'$.
  
 Since $Q=c\times c$, the monodromy map of the $(G^0\times G^0)$-covering $Q$
 is given by two copies of the monodromy map of $c$.
 The covering  $Q$ induces by restriction the $(G^0\times G^0)$-covering
 $F^*\to E'$, whose branch locus is $B'$. 
 Its monodromy map $\mu'\colon\pi_1(E'\setminus B',0')\to G^0\times G^0$ is described  in
 \cite{FP15}.
 
Once $\tau'\in G\setminus G^0$ is fixed, let $\tau:=\tau'^2\in G^0$ and $\varphi\in\Aut(G^0)$ defined by
$\varphi(h):=\tau'h\tau'^{-1}$. We define the following action of $G$ on $G^0\times G^0$:
\[
 \begin{array}{r}
  g(h_1,h_2)=(gh_1,\varphi(g)h_2)\\
  \tau'g(h_1,h_2)=(\varphi(g)h_2,\tau g h_1)
 \end{array}
\qquad \mbox{ for } g\in G^0.
  \]
Finally, we define
\[
 M:=\left|\bigcup_{g\in G}{g\,\im(\mu')}\right|.
\]
\begin{lemma}[{\cite[Lemma 3.2]{FP15}}]\label{degM}
 Let $X:=(C\times C)/G$ be a semi-isogenous mixed surface with $q(X)=1$, then
 $\deg\psi=\displaystyle\frac{|G^0|^2}{M}$.
\end{lemma}

\begin{proposition} \label{genalb}
 Let $X:=(C\times C)/G$ be a semi-isogenous mixed surface with $q(X)=1$, then
\begin{equation}\label{galb}
  g_{alb}=1+M\cdot\dfrac{g(C)-1-|O_2|}{|G^0|^2}\,.
\end{equation}
\end{proposition}

\begin{proof}
 Since $G^0$ is $(1;m_1,\dots,m_r)$-generated, $\displaystyle{e(C)=-|G^0|\sum_{i=1}^r{\left(\frac{m_i-1}{m_i}\right)}}$.
 The $(G^0\times G^0)$-covering $Q$ is branched exactly along the union of $r$ ``horizontal'' copies of $E$ and
 $r$ ``vertical'' copies of $E$; moreover for each $i$ there are one horizontal copy
 and one vertical copy with branching index $m_i$. Since $E'$ is an elliptic curve that
 intersects all of these copies of $E$ transversally in one point, by Hurwitz's formula
 applied to $F^*\to E'$ we get
 \[
  e(F^*)=-|G^0|^2\sum_{i=1}^r{2\left(\frac{m_i-1}{m_i}\right)}=e(C)\cdot |G|\,.
 \]

Let us now consider the map $\eta|_{F^*}\colon F^*\to F$. This map has degree $|G|=2|G^0|$ and by Proposition \ref{ram_loc}
is ramified  in $F^*\cap\left(\bigsqcup_{g\in O_2}{R_{g}}\right)$, and 
\[
\left|F^*\cap\bigsqcup_{g\in O_2} R_g \right|=\sum_{g\in O_2}{\left|F^*\cap R_g\right|}\,.
\]

On one side $R_g=\{(x,(\tau'g)\cdot x): x \in C\}$, hence $Q(R_g)=\{(u,u): u\in E\}$;
on the other side $Q(F^*)=\{(u,- u): u\in E\}$. Therefore, a point  in $\left|F^*\cap R_g\right|$ is 
mapped to a point $(u_0,u_0) \in E'$ with $2u_0=0$.
There are four such points and by assumption none of them lies on the branch curves of $Q$.
For each choice of such a $u_0$, there are exactly $|G^0|^2$ points in $ F^*\cap Q^{-1}(u_0, u_0) $, 
but only $|G^0|$ of them lie on $R_g$, because once we have fixed the first coordinate $x$, the second one is forced to 
be $(\tau' g)\cdot x$.

We get that $\eta|_{F^*}$ is ramified in $\sum_{g\in O_2}{\left|F^*\cap R_g\right|}=|O_2|\cdot 4\cdot |G^0|$ points, and
each one of them has ramification index $2$, thus we have $4\cdot |O_2|$ branching
points. Finally, by Hurwitz's formula,
\[
e(C)\cdot |G|= e(F^*)=|G|\cdot \left(e(F)-\left(\frac{4\cdot |O_2|}{2}\right)\right)=|G|\cdot ( e(F)-2\cdot |O_2|)\,.
\]
By Lemma \ref{degM}, $F$ is the disjoint union of $|G^0|^2/M$ curves of genus $g_{alb}$, therefore
\[
2-2g(C)+2\cdot |O_2|\ =\  e(F)\ =\ \frac{|G^0|^2}{M}(2-2g_{alb}).
\]
\end{proof}

\section{The classification}\label{classif}
In Section \ref{class01}, we classified semi-isogenous mixed surfaces $(C\times C)/G$ with $g(C)=0,1$.
In this section we give an algorithm to classify semi-isogenous mixed surfaces with $g(C)\geq 2$
and fixed values of the invariants: $K_X^2$, $p_g(X)$ and $q(X)$.

\subsection{Finiteness of the classification}
Let  $X:=(C\times C)/G$ be a semi-isogenous mixed surface with $g(C)\geq 2$ and let
$(q;m_1,\ldots,m_r)$ be the type of an induced generating vector for $G^0$.
We define the following rational numbers:
\[
 \Theta:=2q(X)-2+\sum_{i=1}^r{\dfrac{m_i-1}{m_i}}\,,
  \qquad \beta:=\dfrac{2(10\chi(\mathcal O_X)- K^2_X)}{\Theta}\,.
\]

Combining  Proposition \ref{eultop_k2} and Corollary \ref{cor_chi} we  get
\begin{equation}\label{G0_N}
 |G^0|=\frac{(g(C)-1)^2}{10\chi(\mathcal O_X)-K^2_X},
 \qquad |O_2|=\frac{8\chi(\mathcal O_X)-K^2_X}{10\chi(\mathcal O_X)-K^2_X}(g(C)-1).
\end{equation}

\begin{proposition}\label{bounds1}
 Let $X:=(C\times C)/G$ be a semi-isogenous mixed surface with $g(C)\geq 2$
  and let
$(q;m_1,\ldots,m_r)$ be the type of an induced generating vector for $G^0$ . 
Then
\begin{itemize}
\item[(a)] $\Theta>0$ and $\beta=g(C)-1$;
\item[(b)] $r \leq \dfrac{4(10 \chi(\mathcal O_X)-K_X^2)}{\beta} + 4(1-q) $;
\item[(c)] each $m_i$ divides $\beta$;
\item[(d)] $m_i \leq \dfrac{1+2(10\chi(\mathcal O_X)-K^2_X)}{M}$, where $M:=\max\left\{\dfrac{1}{6}, \dfrac{r-3+4q}{2}\right\}$.
\end{itemize}
\end{proposition}

\begin{proof}
a) Since $q(X)=g(C/G^0)$, by Hurwitz's formula:
 $
  2(g(C)-1)=|G^0|\cdot \Theta,
 $
hence $\Theta=\frac{2(g-1)}{|G^0|}>0$. The equation $\beta=g(C)-1$ follows now from equation (\ref{G0_N}).

b) By definition $\Theta\geq 2q-2+\frac{r}{2}$, whence $r\leq 2\cdot \Theta +4(1-q)$.

c) By Riemann's Existence Theorem, there exists an element $h\in G^0$ of order $m_i$ such that $h\cdot x=x$ for some $x\in C$.
Since $h\in G^0$ does not fix any point in $C\times C$, it holds $\varphi(h)\cdot y\neq y$ for all $y\in C$.
Thus the map $C\to C/\langle\varphi(h)\rangle=:\tilde C$ is \'etale of degree $\ord(\varphi(h))=m_i$.
By Hurwitz's formula $2(g(C)-1)=2m_i(g(\tilde C)-1)$.

d) By \cite[Proposition 5.4]{FP15}, $M \leq\Theta +1/m_i$. By part c) we have  $\Theta m_i \leq \Theta \beta$, and 
 by definition of $\beta$ we get
\[M\cdot m_i \leq 1+ \Theta\cdot m_i \leq 1+2(10\chi(\mathcal O_X)-K^2_X).\]
\end{proof}

\begin{remark}\label{finiteness} Let $X:=(C\times C)/G$ be a semi-isogenous mixed surface;
we have an upper bound for $\beta$ given by 
 \[
  \beta=\frac{2(10\chi(\mathcal O_X)-K^2_X)}{\Theta}\leq\frac{2(10\chi(\mathcal O_X)-K^2_X)}{\Theta_{min}},
 \]
where
\[
 \Theta_{min} =
\begin{cases}
1/42 	&  \mbox{ if } q=0, \\
1/2	& \mbox{ if } q=1, \\
2q-2	& \mbox{ if }  q\geq 2.
\end{cases}
\]
This shows that, once we fixed the invariants  $K_X^2$, $p_g(X)$ and $q(X)$, the classification problem
becomes finite.
\end{remark}

We have also the following restrictions for the $m_i$'s.

\begin{proposition}\label{bounds2}
Let  $X:=(C\times C)/G$ be a semi-isogenous mixed surface, 
$(q;m_1,\ldots,m_r)$ be the type of an induced generating vector for $G^0$ and suppose  $|O_2|>0$. 
 
 Then  $m_i \leq |G^0|/|O_2|$ and $m_i$ divides $|O_2|$  for all  $i=1,\dots,r$.
 
\end{proposition}

\begin{proof} 
Let $x_0\in C$ and let $k:=|\Orb(x_0)|$ be the cardinality of its orbit for the $G^0$-action on $C$.
Let $L$ be the curve $L:=\{(x_0,y):y\in C\}$,
the $G$-invariant set $\hat{L}:=\bigcup_{g\in G}{gL}$
 is the union of $2k$ irreducible components, each one isomorphic to $C$: $k$ disjoint ``horizontal'' copies of $C$
  and $k$ disjoint ``vertical'' copies of $C$. 
  Since the action of $G^0$ is free and the elements in $G\setminus G^0$ switch
{horizontal} and {vertical} components of $\hat{L}$,  a ramification point of 
 $\eta|_{\hat{L}} \colon \hat{L}\to\eta(\hat{L})$ belongs to  the set 
 $S:=\{(h_1\cdot x_0, h_2\cdot x_0): h_1,h_2 \in G^0\}$ which has cardinality $k^2$.
On the other side, the ramification locus of $\eta|_{\hat{L}}$ is 
\[\widehat R:=\hat{L}\cap(\bigcup_{g\in O_2}{R_{g}})=\{(h\cdot x_0, (\tau' g )\cdot h \cdot x_0): h\in G^0\,, g\in O_2\}\,,\]
and it has cardinality $k\cdot |O_2|$: 
once  $h\cdot x_0$ is fixed ($k$ choices),  there are $|O_2|$ possibilities for $g$ and each one gives a different point in 
$\widehat R$ by Proposition \ref{ram_loc}.
Thus $k\cdot |O_2|\leq k^2$. 

By Riemann's Existence Theorem,  for each    $i=1,\ldots, r$
there exist
 $ h_i\in G^0$ and  $x_i\in C$ such that $\Stab(x_i)=\langle h_i
 \rangle$ and $\ord(h_i)=m_i$,
 so $k=|\Orb(x_i)|=|G^0|/m_i$. We get $ |O_2| \leq  |G^0|/m_i$.

Now, let   $V:=\{(x_i, (\tau' g)\cdot x_i): g \in O_2 \}$ be the set of ramification points of $\eta$ 
which lie on the vertical line  $\{(x_i,y): y \in C\}$. 
The group $\langle h_i\rangle$ acts faithfully and freely on $V$, 
  indeed 
 \[
  h_i^{\alpha}(x_i, (\tau' g)\cdot x_i)=  (x_i,\varphi(h_i^\alpha)  (\tau' g)\cdot x_i) \neq(x_i, (\tau' g)\cdot x_i) \,, 
  \quad \mathrm{for} \quad \alpha\in\{1,\ldots, m_i-1\}\,,
 \]
 since  $G^0$ acts freely on $C\times C$; whence $m_i $ divides $|V|=|O_2|$.
\end{proof}

\begin{remark}\label{mi_Nj}
Note that the proof of Proposition \ref{bounds2} shows more:  $m_i\leq N_j:=|\Cl(g_j)|$ 
for $i=1,\ldots r$ and $g_j\in O_2$. Indeed, the points in $V$ belong to $m_i$ ramification curves 
 and  $N_j$ is the number of ramification curves mapped onto
the same branch curve.
\end{remark}

\subsection{The Algorithm}

We wrote a MAGMA \cite{MAGMA} script which computes semi-isogenous mixed surfaces (with $g(C)\geq 2$) and fixed 
values of the invariants $p_g$, $q$ and $K^2$.

Here we explain   the strategy of the algorithm, which follows the one used in \cite{FP15};
a commented version of the script can be downloaded from:
\[\mbox{\url{http://www.science.unitn.it/~cancian/}}\,\]

The algorithm  goes as follows: having fixed the values of $K^2$, $p_g$ and $q$, 
by Remark \ref{finiteness}, we have  only finitely many possible types.
Then we  produce the finite list of all types 
$(q;m_1,\ldots,m_r)$ respecting the  conditions in Proposition \ref{bounds1} and  Proposition \ref{bounds2}.

Now, for each types, the orders of $G$ and $G^0$ are computed by 
$|G|=2|G^0|=4(10\chi-K^2)/\Theta^2 $.

Then the script searches, among the finitely many groups of order $|G^0|$, for  groups
having a disjoint pair of generating vector of the prescribed type. 
For these groups, the script checks  their degree $2$ extensions and discards the ones that 
have the wrong number of elements of order 2 and/or do not satisfy the condition of Proposition \ref{bounds2}.

We get a list of quadruples (type, $G^0$, generating vector, extension $G$), each quadruple
gives a family of mixed quotients as explained in Remark \ref{algdata}, 
and all semi-isogenous mixed surfaces with the prescribed invariants are here.
Anyway, in this list there are also surfaces  whose branch locus does 
not correspond to the expected one, then the script  discards them.

Moreover, different generating vectors give deformation equivalent surfaces if they differ by some 
{\it Hurwitz moves} (see \cite{BC04,BCG08}), which are described in great generality in \cite{pen15}. 
The script computes this action on the  generating vectors, and returns only a representative for each 
orbit. Finally, the script computes the fundamental groups of the resulting surfaces,
and, if $q(X)=1$, the genus of the general Albanese fibre too.

\begin{remark}\label{techlim}
The algorithm works for arbitrary values of the invariants $K^2$, $p_g$ and $q$,
 but the implemented MAGMA version has some technical limitations.
  To perform the algorithm, we have to run over all groups of  a given order.
Here we have to use the database of Small Groups, which contains:
\begin{itemize}
\item all groups of order up to 2000, excluding the groups of order 1024;
\item the groups whose order is a product of at most 3 primes;
\item the groups of order dividing $p^6$ for $p$  prime; 
\item  the groups of order $p^n q$, where $p^n$ is a prime-power dividing $2^8$, $3^6$, $5^5$ or $7^4$ and $q$ is a prime different from $p$.
 \end{itemize}
In the other cases we cannot run among the groups of prescribed order and 
the script returns the list of these skipped cases, which have to be studied separately.
\end{remark}

\subsection{Sketch of proof of Theorems \ref{thmA}, \ref{thmB} and \ref{thmC}} \

Let $X:=(C\times C)/G$ be a semi-isogenous mixed surface with $\chi(\mathcal O_X)=1$ and $K_X^2>0$.

If $g(C)\leq 1$,  then $X\cong \mathbb{P}^2$ (see  Section \ref{class01}).

If $g(C)\geq 2$, by Remark \ref{num_br} $X$ has $K_X^2	\leq 8\chi(\mathcal O_X)$,
so the possible values of $K^2_X$ are in $\{1,\dots,8\}$. 

We ran the program for $1\leq K^2\leq 8$ and $0\leq p_g=q\leq 4$.
As expected by the classification results mentioned in the Introduction,
for $p_g=q=4$ the output is empty, while for $p_g=q=3$ we have only one family:
it has $K^2=6$ and is the family of the symmetric products of curves of genus 3, which forms  an 
irreducible connected component of dimension 6 of the moduli space of minimal surfaces of general type.

The other outputs of the program are collected in Tables \ref{q0}, \ref{q1} and \ref{q2}

As mentioned, the surfaces returned by the program may be not all semi-isogenous mixed surfaces
 with the required invariants, since the program  is forced to skip some types, 
 giving rise to groups of large order. The program returns the list of these ``skipped'' cases.
 
For the cases $p_g=q\not =0$, this list is empty. 
We report the list of the ``skipped'' cases for $p_g=q=0$ in Table \ref{tabSkip}.
Some of them can be excluded by arguments very similar to the analogous 
proofs in the papers \cite{BCGP12, BP15, Frap13, FP15}, but for others, the group 
order is too large and these arguments cannot be applied.

\begin{table*}[!t]  \caption{The skipped cases for $p_g=q=0$ and $K^2 > 0$}\label{tabSkip}
{\small
\begin{minipage}{0.3\textwidth}
\begin{tabular}{|c|c|c|}
\hline
  type  &$|G^0|$& $K^2_S$\\
\hline
 2, 3, 7& 21168 &7\\
 2, 3, 8 & 6912 &7\\
 2, 4, 5& 4800 &7\\
 2, 3, 9& 3888 &7\\
 2, 3, 10& 2700&7 \\
\hline
 2, 3, 7& 28224 &6\\
  2, 3, 8& 9216 &6\\
 2, 4, 5& 6400 &6\\
 2, 3, 9&5184 &6\\
 2, 3, 10& 3600 &6\\
 2, 3, 12& 2304 &6\\
 2, 4, 6& 2304 &6\\
 3, 3, 4& 2304 &6\\
 2, 4, 8& 1024 &6\\
 \hline
 2, 3, 7& 35280&5\\
 2, 3, 8& 11520 &5\\
 2, 4, 5& 8000 &5\\
 2, 3, 9& 6480 &5\\
 2, 3, 10& 4500 &5\\
 2, 3, 12& 2880&5\\ 
 2, 4, 6& 2880 &5\\
 3, 3, 4& 2880 &5\\
  \hline
\end{tabular}
     \end{minipage}
     \begin{minipage}{0.3\textwidth}
\begin{tabular}{|c|c|c|}
\hline
  type & $|G^0|$& $K^2_S$ \\
\hline
 2, 3, 7& 42336 &4\\
 2, 3, 8& 13824 &4\\
 2, 4, 5& 9600 &4\\
 2, 3, 9& 7776 &4\\
 2, 3, 10& 5400 &4\\
 2, 3, 12& 3456 &4\\
 2, 4, 6& 3456 &4\\
3, 3, 4& 3456 &4\\
 2, 3, 14& 2646 &4\\
  2, 5, 5& 2400 &4\\
\hline
 2, 3, 7& 49392&3\\
 2, 3, 8& 16128 &3\\
 2, 4, 5& 11200 &3\\
 2, 3, 9& 9072 &3\\
 2, 3, 10& 6300 &3\\ 
 2, 3, 12& 4032 &3\\
 2, 4, 6& 4032 &3\\
 3, 3, 4& 4032 &3\\
 2, 5, 5& 2800 &3\\
 2, 3, 18& 2268 &3\\
\hline
\end{tabular}
     \end{minipage}     
     \begin{minipage}{0.3\textwidth}
\begin{tabular}{|c|c|c|}
\hline
  type  &$|G^0|$& $K^2_S$\\
\hline
 2, 3, 7& 56448 &2\\
 2, 3, 8&18432 &2\\
 2, 4, 5& 12800 &2\\
 2, 3, 9& 10368 &2\\ 
 2, 3, 10& 7200 &2\\
 2, 3, 12& 4608 &2\\
 2, 4, 6& 4608 &2\\
 3, 3, 4& 4608 &2\\
 2, 3, 14& 3528 &2\\
 2, 5, 5& 3200 &2\\
 2, 3, 18& 2592 &2\\
 2, 4, 8& 2048 &2\\

\hline
 2, 3, 7& 63504 &1\\
 2, 3, 8& 20736 &1\\
 2, 4, 5& 14400 &1\\  
 2, 3, 9& 11664 &1\\  
 2, 3, 10& 8100 &1\\
 2, 3, 12& 5184 &1\\ 
 2, 4, 6& 5184 &1\\
 3, 3, 4& 5184 &1\\   
 2, 5, 5& 3600 &1\\
 2, 3, 18& 2916 &1\\ 
 2, 4, 8& 2304 &1\\
 3, 3, 5& 2025 &1\\
  \hline
\end{tabular}
     \end{minipage}}
   
\end{table*}

Now let us consider the  surfaces in Tables \ref{q0}, \ref{q1} and \ref{q2}.
 A surface with $K^2>0$ is either of general type or rational, 
therefore regular and simply connected: a quick inspection of the tables shows that this latter case does not occur, so all constructed surfaces are of general type.

\section{On the minimality}\label{On_min}

In this section we address the minimality problem for the surfaces listed in Tables \ref{q0}, \ref{q1} and \ref{q2}.

Let  $X:=(C\times C)/G$ be a semi-isogenous mixed surface with $g(C)\geq 2$, and let
$Y:=(C\times C)/G^0$. We denote by $\pi\colon Y\to X$ the covering map, by $\mathcal{B}$ the branch locus of $\pi$ 
and by $\mathcal{R}$ its ramification locus.

\begin{remark}\label{no01curve}

Let $D$ be a (possibly singular) irreducible curve on $Y$.
Let $\tilde{D}$ be the normalization of $D$: there exists a proper map
 $\nu:\tilde{Y}\to Y$ consisting of a finite number of blow-ups 
such that  the strict transform $\tilde{D}$ of $D$ is smooth (see \cite[Theorem II.7.1]{BHPV}).
 We have the following commutative diagram
\begin{equation}\label{norm}
 \xymatrix{
C\times C\ar[d]_\sigma &(C\times C)\times_Y\tilde{Y} \ar[l]_{\gamma_1\qquad}\ar[d]^{\gamma_2}  \\
Y& \tilde{Y}\ar[l]_\nu}
\end{equation}
where $(C\times C)\times_Y\tilde{Y}$ denotes the fiber product and $\gamma_1$ and $\gamma_2$ the 
natural projections; the map  $\gamma_2$ is \'etale because $\sigma$ is \'etale.
Let $D'$ be an irreducible component of  $\gamma_2^{-1}(\tilde{D})$, 
its image $\gamma_1 (D')$ is a curve in $C\times C$, and therefore surjects onto $C$, whence
$g(D')\geq g(C)\geq 2$.  Since  $\gamma_2$ is \'etale, and $D'$ and $\tilde D$ are both smooth, 
 we deduce that  $g(\tilde D) \geq 2$.
\end{remark}

Let $E$ be a smooth rational curve on $X$. %
If $E\cap\mathcal{B}=\emptyset$, then there exists a rational curve in $\pi^{-1}(E)\subset Y$, 
 contradicting Remark \ref{no01curve}.
 By Proposition \ref{br_loc}, $\mathcal{B}$ is a finite union of disjoint curves of genus 
strictly greater than $1$, 
hence $E$ and $\mathcal{B}$ meet in a finite number of points. 
We split these points in two sets accordingly to the parity of their intersection multiplicity:
\[\begin{array}{rlc}
A_0&:=&\{p\in E\cap \mathcal{B}:m_p(E\cap\mathcal{B})\text{ is even}\},\\
A_1&:=&\{p\in E\cap \mathcal{B}:m_p(E\cap\mathcal{B})\text{ is odd}\},
\end{array}\]
where $m_p(E\cap\mathcal{B})$ denotes the intersection multiplicity of $E$ and $\mathcal{B}$ in $p$. 
We define $\mu_0:=|A_0|, \mu_1:=|A_1|$ and $\mu:=\mu_0+\mu_1=|E\cap\mathcal{B}|$.

\begin{lemma}
\label{d1_even} Let  $X:=(C\times C)/G$ be a semi-isogenous mixed surface with $g(C)\geq 2$ and 
$E\subset X$ be a smooth
 rational curve. Then $\mu_1$ is even and $\mu_1\geq 6$.
\end{lemma}

\begin{proof}
Let $D:=\pi^{-1}(E)\subset Y$. By Hurwitz's formula $e(D)=2e(E)-\mu=4-\mu$. 
The map $\pi|_D:D\to E$ is finite of degree $2$. In particular, if $D$ is reducible, $D=D_1+D_2$ with $\pi|_{D_i}:D_i\to E$ biregular, contradicting Remark \ref{no01curve}. So $D$ is irreducible.

Since $\pi$ is a local isomorphism out of $\mathcal{R}$, the singularities of $D$ lie on $\mathcal{R}$. 
Let us fix $p\in E\cap\mathcal{B}$, and suppose $m_p(E\cap\mathcal{B})=k$, then 
we can take local coordinates $(x,y)$ centred in  $p$ such that $\mathcal{B}=\{x=0\}$ and $E=\{x=y^k\}$.
 Denoting by $(z,w)$ local coordinates  centred in  $p'=\pi^{-1}(p)$, the local expression of the map 
 $\pi\colon Y\to X$ is $(z,w)\mapsto(z^2,w)$:  $\mathcal{R}=\{z=0\}$ and $D=\{z^2=w^k\}$. 
 This means that $p'$ is a singular point if and only if $k\geq 2$.

The blow-up of $Y$ in $p'$ is given on a chart (say $V_1$) by $(x_1, y_1)\mapsto(x_1y_1, y_1)$, 
and  on the other chart (say $W_1$)  by
 $(u_1, v_1)\mapsto(u_1, u_1v_1)$ and the glueing $V_1\to W_1$ is given by 
 $(x_1, y_1) \mapsto (x_1y_1, x_1^{-1})$.
So the strict transform $D_1$ of $D$ on $V_1$ is $\{x_1^{2}=y_1^{k-2}\}$, while on $W_1$ is given by $\{u_1^{k-2}v_1^k=1\}$. 

If $k-2\geq 2$ then $D_1$ has a singular point and we blow up again, otherwise $D_1$ is smooth and we stop.
According to the parity of $k$,  we get eventually either  $x_n^2=1$ or $x_n^2=y_n$; thus, if $k$ is even  there are two points on the strict transform $D_n$ lying over $p'$, otherwise there is only one point over $p'$. 
Repeating this process for each  singular points of $D$, we get that for the normalization $\tilde D$ of $D$ it holds $e(\tilde{D})=e(D)+\mu_0=4-\mu_1$, whence  $2-2g(\tilde{D})=4-\mu_1$, i.e.~$\mu_1=2g(\tilde{D})+2\geq 6$,
where the inequality follows by Remark \ref{no01curve}.
\end{proof}

\begin{proposition}\label{EB}
Let  $X:=(C\times C)/G$ be a semi-isogenous mixed surface with $g(C)\geq 2$ and 
$E\subset X$ be a smooth
 rational curve.  Then  $E.\mathcal{B}$ is even and $E.\mathcal{B}\geq 6$.
\end{proposition}

\begin{proof}

Being
\[E.\mathcal{B}=\sum_{p\in E\cap\mathcal{B}}{m_p(E\cap\mathcal{B})}=\sum_{p\in A_0}{m_p(E\cap\mathcal{B})}+\sum_{p\in A_1}{m_p(E\cap\mathcal{B})},
\]
by Lemma \ref{d1_even}, we get the claim.
\end{proof}

We recall the following.

\begin{theorem}[Hodge Index Theorem, {\cite[Corollary IV.2.16]{BHPV}}]
\label{hit}
Let $S$ be a smooth surface, $NS(S)$ be its Neron-Severi group and consider 
$NS(S)\otimes_{\mathbb{Z}}\mathbb{R}$ endowed with the quadratic form induced by the intersection product. 
Let $D$ be a divisor on $S$ with $D^2>0$. 
Then the intersection product is negative definite on the orthogonal complement $D^\perp$ of $D$
 in $NS(S)\otimes_{\mathbb{Z}}\mathbb{R}$.
\end{theorem}

For a divisor $D$ on $S$, we denote by $[D]$ its class
 in the vector space $NS(S)\otimes_{\mathbb{Z}}\mathbb{R}$.

\begin{lemma}[{\cite[Remark 4.3]{BP12}}]\label{BP}
On a smooth surface $S$ of general type every irreducible curve $C$
with $K_S.C\leq 0$ is smooth and rational.
\end{lemma}

\begin{proposition}
\label{67min}
Let $X:=(C\times C)/G$ be a semi-isogenous mixed surface of general type with invariants $\chi(\mathcal{O}_X)=1$
and $K^2_X>0$, and let $\rho\colon X\to X_{min}$ be  the projection to its minimal model.
Then
\begin{itemize}
\item for $K^2_X\in\{6,7,8\}$, $\rho$ is the identity map: $X=X_{min}$;
\item for $K^2_X\in\{4,5\}$, $\rho$ is the contraction of at most one $(-1)$-curve;
\item for $K^2_X\in\{2,3\}$, $\rho$ is the contraction of at most two $(-1)$-curves;
\item for $K^2_X=1$, $\rho$ is the contraction of at most three $(-1)$-curves.
\end{itemize}
\end{proposition}

\begin{proof} 
Let $E_1\subset X$ be a $(-1)$-curve. By Lemma \ref{BP}, $E_1$ is smooth, hence $n_1:=E_1.\mathcal{B}\geq 6$.
Let $W$ be the subspace of $NS(X)\otimes_{\mathbb{Z}}\mathbb{R}$ generated by $[K_X],[\mathcal{B}],[E_1]$.
Let us consider the matrix
\[
M_1:=
\begin{pmatrix}
 K_X^2 &  K_X.\mathcal B&  K_X.E_1 \\
 K_X.\mathcal B & \mathcal B^2 & E_1. \mathcal B \\
 K_X.E_1&  E_1.\mathcal B&E_1^2
\end{pmatrix}
=
\begin{pmatrix}
 K_X^2 &  6(8-K_X^2)&-1 \\
 6(8-K_X^2) & -4(8-K_X^2) & n_1 \\
  -1&n_1 & -1
\end{pmatrix}\,,
\]
it has determinant $\det M_1= -K_X^2 n_1^2-12(8 -K_X^2)n_1+4(8-K_X^2)(73-8K^2_X)$.
As quadratic polynomial in $n_1$, $\det M_1$ has roots ($x_1\leq x_2$):
\[ x_{1,2}= \dfrac{-6(8-K_X^2)\pm \sqrt{36(8-K_X^2)^2+K_X^2(8-K_X^2)(73-8K_X^2)}}{K_X^2}
\]
and it is easy to see that $x_1\leq 0 \leq x_2$.  
Since $K_X^2>0$, by  Theorem \ref{hit} we have $ \det M_1\geq 0$,
and the leading coefficient $-K_X^2$ is negative hence $6\leq n_1\leq \lfloor x_2\rfloor$.
For $K^2_X>0$, the round-down  of  $x_2$  is: \setlength{\arraycolsep}{5pt}
\begin{equation}\label{N1}
\begin{array}{c|c|c|c|c|c|c|c|c}
K_X^2 & 8&7&6&5&4&3&2&1\\ \hline
\lfloor x_2\rfloor & 0  &2&4 & 6& 8&10 &13&17
\end{array}
\end{equation}
For $K_X^2\in \{6,7,8\}$, it holds $\lfloor x_2\rfloor< 6$, a contradiction, whence 
there are no $(-1)$-curves on $X$.

Assume now $K^2_X<6$ and assume there exists another rational curve  $E_2\subset X$ such that either $E_2^2=-2$ and $E_1.E_2=1$ or $E_2^2=-1$ and   $E_1.E_2=0$.  Up to change $E_2$ with $E_1+E_2$,
the matrix of the intersection form  for  $[K],[\mathcal{B}],[E_1],[E_2]$ is
\[
M_2:=
\begin{pmatrix}
 K_X^2 &  K_X.\mathcal B&  K_X.E_1 &  K_X.E_2\\
 K_X.\mathcal B & \mathcal B^2 & E_1. \mathcal B&   E_2. \mathcal B\\
 K_X.E_1&  E_1.\mathcal B&E_1^2&  E_1.E_2\\
  K_X.E_2&  E_2. B&E_1.E_2& E_2^2
\end{pmatrix}
=
\begin{pmatrix}
 K_X^2 &  6(8-K_X^2)&-1 &-1\\
 6(8-K_X^2) & -4(8-K_X^2) & n_1&n_2 \\
  -1&n_1 & -1&0\\
  -1&n_2& 0 &-1
\end{pmatrix},
\]
where $n_2:=E_2.\mathcal{B}\geq 6$. \\ It has   	
$\det M_2=n_2^2 (1+K^2_X)+n_2 (12(8-K^2_X)-2n_1)+n_1^2(1+K^2_X)+12n_1(8-K^2_X)+8(8-K^2_X)(4K^2_X-37)$.\\
As quadratic polynomial in $n_2$, $\det M_2$ has roots $y_1\leq y_2$.

 Since $K_X^2>0$, by  Theorem \ref{hit} we have $ \det M_2\leq 0$,
and the leading coefficient $1+K_X^2$ is positive hence $6\leq n_2\leq\lfloor y_2\rfloor$.
For $6>K^2_X>0$,  and $n_1\geq 6$ even (Lemma \ref{BP}) and bounded from above by the value
in (\ref{N1}), $y_1$ is negative and the round-down of  $y_2$  is: 
\[
\begin{array}{c|c|c|c|c|c|c|c|c|c|c|c|c|c|c|c|c}
K_X^2 & 5&\multicolumn{2}{c|}{4}&\multicolumn{3}{c|}{3}&\multicolumn{4}{c|}{2}&\multicolumn{6}{c}{1}\\ \hline
n_1 & 6&6&8&6&8&10&6&8&10&12&6&8&10&12&14&16\\
\lfloor y_2\rfloor & -2& 3&-2&6&4& 0&9&7&5&1&12&11&9&7&4&-1\\
\end{array}
\]
For $K_X^2\in \{4,5\}$, it holds $\lfloor y_2\rfloor < 6$, a contradiction, whence 
there is at most one $(-1)$-curve on $X$.

Arguing in the same way, one proves the statements in the remaining cases: $K^2_X\in\{1,2,3\}$.
\end{proof}

\begin{corollary}
Let $X$ be a semi-isogenous mixed surface of general type with $K^2_X=2$ and 
$p_g(X)=q(X)=2$.  Its minimal model $X_{min}$ has  $K^2_{X_{min}}=4$.
\end{corollary}

\begin{proof}
By Debarre's inequality (see \cite{deb82}),
for a minimal irregular surface  of general type $S$ it holds $K^2_S\geq 2 p_g(S)$, thus $X$ is not minimal and
$K^2_{X_{min}}\geq 4$, i.e.~we need to contract at least two $(-1)$-curves.

On the other side, by Proposition \ref{67min} we can  contract at most 
two $(-1)$-curves.
\end{proof}

\subsection{The case $p_g(X)=q(X)=2$ and $K_X^2=4$}
Let $X:=(C\times C)/G$ be a semi-isogenous mixed surface with $K_X^2=4$ and $p_g(X)=q(X)=2$.
The surface $X$ is of general type and, according to Table \ref{q2}, $C$ is a curve of genus $7$, 
$G^0\cong S_3$ and  $G\cong S_3\times\mathbb{Z}_2$: $\mathbb{Z}_2$ acts on $C\times C$ exchanging the
factors. 

By Proposition \ref{67min}, $X$ contains at most a $(-1)$-curve.  We explicitly construct  a $(-1)$-curve on $X$,
thus $X_{min}$ has $K^2_{X_{min}}=5$.

By Proposition \ref{irreg}, $C':=C/G^0$ is a curve of genus 2: it is  hyperelliptic. 
Let $f'\colon C'\to C'$ be the hyperelliptic involution and $c\colon C\to C'$ the projection. 
According to  \cite[Corollary 2]{Acc}, $f'$ lifts to an automorphism $f\in\Aut(C)$,
 i.e.~$f$ satisfies  $c(f(p))=f'(c(p))$ for all $p\in C$. \\
By the uniqueness of the lift, $faf\in S_3$ for all $a\in S_3$; in particular $f^2\in S_3$ and $S_3\triangleleft H:=\langle S_3,f\rangle<\Aut(C)$
 with $H$ of order $12$. The map $C\to C/H\cong\mathbb{P}^1$ ramifies in $36$ points, each one of them with stabilizer generated 
 by an element of order $2$ in $H\setminus S_3$. Let $T:=\{p_i\}_{i=1,\dots,36}$ be the ramification locus of $C\to C/H$.

Since $\Aut(S_3)=\Inn(S_3)$, there exists a unique $\bar{f}\in H\setminus S_3$ such that $\bar{\varphi}(g):=\bar{f}g\bar{f}^{-1}=\varphi(g)$ for all $g\in S_3$. Let $\Gamma:=\{(x,\bar{f}x):x\in C\}\subset C\times C$ be the graph of $\bar{f}$. A direct computation shows that the curve $\Gamma$ is $G$-invariant. Let $\tilde{\Gamma}:=\eta(\Gamma)\subset X$ ; the ramification locus of the map $\eta|_{\Gamma}:\Gamma\to\tilde{\Gamma}$ is $\{(\bar{f}p,\bar{f}^2p):p\in T\}$, and each ramification point has stabilizer of order $2$ generated by an element of $G\setminus G^0$, then, by Hurwitz's formula,
\[
12=2g(\Gamma)-2=12(2g(\tilde{\Gamma})-2)+36,
\]
that is $g(\tilde{\Gamma})=0$.

The canonical divisor $K_{C\times C}$ is numerically equivalent to 
\[2(g(C)-1)F_1+2(g(C)-1)F_2=12F_1+12F_2\,,\] 
where $F_1,F_2$ denote a general fiber of the projections onto the first and onto the second coordinate respectively, 
then $K_{C\times C}.\Gamma=24$. 
By the adjunction formula \[\Gamma^2=2g(\Gamma)-2-K_{C\times C}.\Gamma=-12\,.\]
Finally, since $\eta^*(\tilde{\Gamma})=\Gamma$, the projection formula $\Gamma^2=\eta^*(\tilde{\Gamma})^2=\deg\eta\cdot\tilde{\Gamma}^2$ implies $\tilde{\Gamma}^2=-1$.

\subsection{The cases $p_g(X)=q(X)=0$ and $K_X^2=2$}

Let $X$ be one of the surfaces in Table \ref{q0} with $K_X^2=2$. The order of $H_1(X, \mathbb Z)$ is $32$.

  For a \textit{numerical Campedelli surface} $S$, i.e.~a minimal surface of general type with
  $K^2_S=2$ and $p_g(S)=0$, it is known (cf. \cite{Rei}) that  its  algebraic  fundamental  group 
  is a finite group of order $\leq 9$.
  
As remarked in \cite{BCPsurvey} (see also \cite{PPS10}),
if $H_1(S, \mathbb Z)$ is finite,  it  is isomorphic to the abelianization
of   $\pi_1^{alg}(S)$ and so it has order $\leq 9$,
whence $X$ cannot be minimal.

\subsection{The cases $p_g(X)=q(X)=1$ and $K_X^2=2$}
By \cite{Cat81}, the minimal surfaces of general type with $p_g=q=1$ and $K^2=2$ form a connected component in the moduli space: 
the Albanese map of these surfaces is a genus $2$ fibration, and their fundamental group is isomorphic to $\mathbb Z^2$ (cf.~\cite{FP15}).
We conclude  that  the surfaces in Table \ref{q1} with $K^2=2$  are not minimal.

\noindent{\bf Authors' Adresses:}

{\it Nicola  Cancian}: Dipartimento di Matematica, Universit\`a degli Studi di Trento;\\
Via Sommarive 14; I-38123 Povo (Trento), Italy

{\it Davide Frapporti}: Lehrstuhl Mathematik VIII,  Universit\"at Bayreuth;\\
Universit\"atsstra\ss e 30; D-95447 Bayreuth, Germany

\end{document}